\documentclass[a4paper,12pt, reqno]{amsart}
\usepackage{amsfonts, color}
\usepackage{geometry}
\usepackage[foot]{amsaddr}

\usepackage{amstext, amssymb, amsmath}
\usepackage{graphicx}
\usepackage{color}
\usepackage[pagewise]{lineno}
\usepackage[linesnumbered,lined,ruled,commentsnumbered]{algorithm2e}

\numberwithin{equation}{section}

\theoremstyle{plain}
\newtheorem{theorem}{Theorem}[section]
\newtheorem{lemma}[theorem]{Lemma}
\newtheorem{proposition}[theorem]{Proposition}
\newtheorem{corollary}[theorem]{Corollary}

\theoremstyle{definition}
\newtheorem{defn*}{Definition}
\newtheorem{definition}{Definition}
\newtheorem{remark}[theorem]{Remark}

\newtheorem{assumption}{Assumption}

\newcommand{\R}{\mathbb{R}}
\newcommand{\F}{\mathcal{F}}
\newcommand{\U}{\mathcal{U}}

\newcommand{\e}{\varepsilon}
\newcommand{\weak}{\rightharpoonup}

\newcommand{\ensemb}{\mathcal{E}_{\Theta}}
\newcommand{\ensembM}{\mathcal{E}_{M}}

\title[Deep Learning of Diffeomorphisms]
{Deep Learning Approximation of Diffeomorphisms via Linear-Control Systems}
\author[A. Scagliotti]{A. Scagliotti}

\address[A.~Scagliotti]{Scuola Internazionale Superiore di Studi Avanzati, Trieste, Italy.}
\email{ascaglio@sissa.it}

\begin{document}

\begin{abstract}
In this paper we propose a Deep Learning architecture
to approximate diffeomorphisms diffeotopic to the
identity. 
We consider a control system
of the form $\dot x = \sum_{i=1}^lF_i(x)u_i$, with linear dependence
in the controls, and we use the corresponding flow
to approximate the action of a diffeomorphism on 
a compact ensemble of points. 
Despite the simplicity of the
control system, it has been recently shown that
a Universal Approximation Property holds.
The problem of minimizing the sum of the training 
error and of a regularizing term  
induces a gradient flow in the space of admissible
controls.
A possible training procedure
for the discrete-time neural network
consists in projecting
the gradient flow onto a finite-dimensional
subspace of the admissible controls.
An alternative approach
relies on an iterative method
based on Pontryagin Maximum Principle
for the numerical resolution of Optimal
Control problems. Here the maximization of the
Hamiltonian can be carried out with
an extremely low computational effort,
owing to the
linear dependence of the system in the
control variables.
Finally, we use tools from $\Gamma$-convergence to
provide an estimate of the expected generalization
error.  
\subsection*{Keywords}
Deep Learning,
Linear-Control System, $\Gamma$-convergence,
Gradient Flow, Pontryagin Maximum Principle.
\subsection*{MSC2020 classification}
Primary: 49M05, 49M25. Secondary: 68T07, 49J15.
\end{abstract}

\maketitle

\begin{section}{Introduction}
{Residual Neural Networks (ResNets)
 were originally introduced
in \cite{HZRS16} in order to overcome some 
issues related to the training process of traditional
Deep Learning architectures. Indeed, it had been
observed that, as the number of the layers in 
non-residual architectures is increased,
the learning of the parameters is affected by the 
vanishing of the gradients (see, e.g., \cite{BSF94})
and the accuracy of the network gets rapidly 
saturated (see \cite{HS15}).}
%
%
%
ResNets can be represented as the
composition of non-linear mappings
\begin{equation*} 
\Phi = \Phi_N \circ\ldots \circ \Phi_1, 
\end{equation*}
where $N$ represents the {\it depth} of the Neural
Network and, for every $k=1\ldots N$,
the {\it building blocks} $\Phi_k:\R^n \to \R^n$ are of the form
\begin{equation} \label{eq:ResNet_intro}
\Phi_k(x) = x + \sigma (W_k x + b_k),
\end{equation}
where $\sigma:\R^n \to \R^n$
is a non-linear activation function
that acts component-wise, and $W_k\in \R^{n\times n}$
and $b_k\in \R^n$ are the parameters to be learned.
{In contrast, we recall that in 
non-residual architectures $\bar \Phi = \bar \Phi_N\circ\ldots \circ \bar \Phi_1$, the 
{building blocks} have usually the form
\[
\bar \Phi_k(x) = \sigma (W_k x +b_k)
\]
for $k=1,\ldots,N$.} 
{In some
recent contributions 
\cite{E17,LCTE17,HR17},
ResNets have been studied in the framework of 
Mathematical Control Theory.
The bridge between ResNets and
Control Theory was independently discovered
in \cite{E17} and
\cite{HR17}, where it was observed
that each function $\Phi_1,
\ldots,\Phi_N$ defined as in \eqref{eq:ResNet_intro}
can be interpreted as an Explicit Euler} 
discretization of the control system
\begin{equation} \label{eq:ctrl_sys_ResNet}
\dot x = \sigma(Wx + b),
\end{equation}
where $W$ and $b$ are the control variables.
{Since then, Control Theory has been
fruitfully applied to the theoretical understanding
of ResNets.} 
In \cite{TG20} a Universal 
Approximation result for the flow generated by
\eqref{eq:ctrl_sys_ResNet} was established
under suitable assumptions on the activation function
$\sigma$.
In \cite{LCTE17} and \cite{BC19} the problem 
of learning an
unknown mapping was formulated as an Optimal
Control problem, and the Pontryagin Maximum Principle
was employed in order to learn the optimal parameters.
In \cite{BCFH21} it was consider the mean-field
limit of \eqref{eq:ctrl_sys_ResNet}, and it was
proposed a training algorithm based on the
discretization of the necessary optimality conditions
for an Optimal Control problem in the space
of probability measures.
The Maximum Principle for Optimal Control of 
probability measures was first introduced in 
\cite{P04}, and recently it has been independently 
rediscovered in \cite{BFRS17}.
{In this paper, rather than using 
tools from Control Theory to study properties of 
existing ResNets, we propose an architecture
inspired by theoretical observations
on control systems with {\it linear} dependence
in the control variables. As a matter of fact, the
building blocks of the ResNets that we shall
construct depend linearly in the parameters,
namely they have the form
\begin{equation*}
\Phi_k(x) = x + G(x) u_k,
\end{equation*}
where $G :\R^n\to\R^{n\times l}$
is a non-linear function of the input, and 
$u_k\in \R^l$ is the vector of the parameters at
the $k$-th layer.} 
The starting points of this paper are the 
controllability results proved in \cite{AS1, AS2},
where the authors considered a control system
of the form
\begin{equation} \label{eq:lin_ctrl_sys_intro}
\dot x =F(x)u= \sum_{i=1}^lF_i(x)u_i,
\end{equation}
where $F_1,\ldots,F_l$ are  smooth and bounded
vector fields on $\R^n$,
and $u\in \U:= L^2([0,1],\R^l)$ is the control.
We immediately observe that
\eqref{eq:lin_ctrl_sys_intro} 
has a simpler structure
than \eqref{eq:ctrl_sys_ResNet},
having linear dependence
with respect to the control variables.
Despite this apparent simplicity, the flows
associated to \eqref{eq:lin_ctrl_sys_intro}
are capable of interesting approximating results.
Given a control $u\in \U$, let 
$\Phi_u:\R^n\to\R^n$ be the flow obtained 
by evolving \eqref{eq:lin_ctrl_sys_intro} on the
time interval $[0,1]$ using the control $u$.
In \cite{AS2} it was formulated a condition
for the controlled vector fields $F_1,\ldots,F_l$
called {\it Lie Algebra Strong Approximating
Property}. On one hand, this condition guarantees
{exact controllability} on finite ensembles, i.e., 
for every $M\geq 1$,
for every $\{x^j_{0}\}_{j=1,\ldots,M}\subset\R^n$
such that
$j_1\neq j_2\implies x^{j_1}_{0}\neq x^{j_2}_{0}$,
 and for every 
injective mapping $\Psi:\R^n\to\R^n$, there exists
a control $u\in \U$ such that
$\Phi_u(x^j_{0})=\Psi(x^j_{0})$ for every $j=1,\ldots,M$. 
On the other hand, this property is also a sufficient 
condition for a $C^0$-approximate controllability
result in the space of diffeomorphisms.
More precisely, given a diffeomorphism 
$\Psi:\R^n\to\R^n$ diffeotopic to the identity, and
given a compact set $K\subset\R^n$, for every 
$\e>0$ there exists a control 
$u_\e\in \U$ such that
$\sup_K|\Phi_{u_\e}(x)-\Psi(x)|\leq \e$.
The aim of this paper is to use this machinery
to provide implementable
algorithms for the approximation of diffeomorphisms
diffeotopic to the identity.
{More precisely, we discretize the
control system \eqref{eq:lin_ctrl_sys_intro} 
on the evolution interval $[0,1]$ using the 
Explicit Euler scheme and the uniformly
distributed time-nodes $\{ 0,\frac{1}{N},
\ldots, \frac{N-1}N,1 \}$, and we obtain} the 
ResNet represented by the composition 
$\Phi=\Phi_N \circ \ldots \circ \Phi_1$,
where, for every 
{$k=1,\ldots,N$, $\Phi_k$ is of
the form
\begin{equation}\label{eq:ResNet_lin_intro}
x_{k} =
\Phi_k(x_{k-1}) = x_{k-1} 
+ h\sum_{i=1}^lF_i(x_{k-1})u_{i,k},
\quad h=\frac1N,
\end{equation}
where $x_0\in \R^n$ represents an initial input 
of the network. In this construction, the points
$(x_k)_{k=0,\ldots,N}$ represent the approximations
at the time-nodes $\{\frac{k}{N}\}_{k=0,\ldots,N}$
 of the  
trajectory $x_u:[0,1]\to\R^n$
that solves \eqref{eq:lin_ctrl_sys_intro}
with Cauchy datum $x_u(0)=x_0$.
We insist on the fact that in our discrete-time model
we assume that the controls are piecewise constant 
on the time intervals $\left\{ [\frac{k-1}{N}, 
\frac{k}{N}) \right\}_{k=1,\ldots,N}$. For this 
reason, when we derive a ResNet with $N$ hidden 
layers, we deal with $N$ building blocks 
(i.e., $\Phi_1,\ldots,\Phi_N$) and with $N$
$l$-dimensional parameters 
$(u_k)_{k=1,\ldots,N}
=(u_{i,k})_{k=1,\ldots,N}^{i=1,\ldots,l}$. 
Moreover, when we evaluate the ResNet at
a point $x_0\in \R^n$, we end up with
$N+1$ input/output variables 
$(x_k)_{k=0,\ldots,N}\subset \R^n$.} 


{As anticipated before, we use the 
ResNet \eqref{eq:ResNet_lin_intro} for the 
task of a data-driven reconstruction of 
diffeomorphisms.}
The idea is that we may imagine to observe the action
of an unknown diffeomorphism $\Psi:\R^n \to \R^n$
on a finite ensemble of points $\{ x^{j}_{0}\}_{j=1,
\ldots,M}$ with $M\geq 1$. Using these data,
we try to approximate the mapping $\Psi$.
Assuming that the controlled vector fields 
$F_1,\ldots,F_l$ satisfy the Lie Algebra Strong
Approximating Property, 
a first natural attempt consists in exploiting
the exact controllability
for finite ensembles. Indeed, this ensures that
we can find 
a control $u^M\in \U$ that achieves a null training
error, i.e., the
corresponding flow $\Phi_{u^M}$ satisfies
$\Phi_{u^M}(x^j_{0})
=\Psi(x^j_{0})$ for every $j=1,\ldots,M$.
Unfortunately, when the number of observation $M$
grows, this training strategy
can not guarantee any theoretical upper
bound for the generalization error.
{Indeed, as detailed 
in Section~\ref{sec:def_strat},
in the case $\Psi$ is not itself a flow of
the control system \eqref{eq:lin_ctrl_sys_intro},
 if we pursue the {\it null-training error} 
strategy then
the sequence of the controls $(u^M)_{M\geq1}$
will be unbounded in the $L^2$-norm. Since the 
Lipschitz constant of a flow generated
by \eqref{eq:lin_ctrl_sys_intro} is related to
the $L^2$-norm of the corresponding control, 
it follows that we cannot bound from above the
Lipschitz constants of the approximating
diffeomorphisms $(\Phi_{u^M})_{M\geq1}$. 
Therefore, even though for every $M\geq1$ we achieve
perfect matching $\Phi_{u^M}=\Psi$ on the training
dataset $\{ x_0^j \}_{j=1,\ldots,M}$, we have
no control on the generalization error 
made outside the training points. 
This is an example of {\it overfitting},
a well-known phenomenon in the practice of 
Machine Learning.}

In order to overcome this issue,
{when constructing the approximating 
diffeomorphism we need to find a balance between
the training error and the $L^2$-norm of the 
control.
For this reason} we consider
the non-linear functional $\F^M:\U\to\R$
defined on the space of admissible 
controls as follows:
\begin{equation} \label{eq:cost_intro}
\F^M(u) := \frac1M
\sum_{j=1}^M a(\Phi_u(x^{j}_0) - \Psi(x^{j}_0))
+ \frac{\beta}{2}||u||_{L^2}^2,
\end{equation}
where $a:\R^n \to \R$ is a non-negative function
such that $a(0)=0$, and $\beta>0$ is a fixed 
parameter that tunes the squared $L^2$-norm
regularization. A training strategy alternative to
the one described above consists in looking
for a minimizer $\tilde u^M$ of 
$\F^M$, and taking $\Phi_{\tilde u^M}$ as an
approximation of $\Psi$.
Under the assumption that
the training dataset is sampled from
a compactly-supported
Borel probability measure $\mu$ on $\R^n$,
we prove a $\Gamma$-convergence result.
Namely, we show that, as the size $M$
of the training dataset  tends to infinity,
the problem
of minimizing $\F^M$ {\it converges} to the problem
of minimizing the functional $\F:\U\to\R$, defined as
\begin{equation}\label{eq:G_lim_intro}
\F(u) := \int_{\R^n} a(\Phi_u(x)-\Psi(x)) \,d\mu(x)
+ \frac\beta2 ||u||_{L^2}^2.
\end{equation}
Using this $\Gamma$-convergence result, we deduce the
following upper bound for the expected testing
error:
\begin{equation} \label{eq:test_est_intro}
\mathbb{E}_\mu [a(\Phi_{\tilde u^M}(\cdot)-
\Psi(\cdot))] \leq 2\kappa_\beta + 
C_{a,\Psi,\beta}W_1(\mu_M,\mu),
\end{equation}
where $\tilde u^M$ is a minimizer of $\F^M$,
$\kappa_\beta$ and $C_{a,\Psi,\beta}$ 
are constants such that 
$\kappa_\beta \to 0$ 
and $C_{a,\Psi,\beta}\to\infty$
 as $\beta \to 0$, and
$W_1(\mu_M,\mu)$ is the Wasserstein distance 
between $\mu$ and the empirical measure
\[
\mu_M:= \frac1M \sum_{j=1}^M\delta_{x^j}.
\]
Recalling that $W_1(\mu_M,\mu)\to 0$ as 
$M\to \infty$, inequality
\eqref{eq:test_est_intro} ensures 
that we can achieve arbitrarily small expected 
testing error by choosing $\beta$ small enough and
having a large enough training dataset.
{Indeed, assuming that we can choose the 
size of the dataset, for every $\varepsilon>0$ 
we consider $\beta >0$ such that 
$\kappa_\beta\leq\frac{\e}{2}$, and then we take
$M$ such that $C_{a,\Psi,\beta}W_1(\mu_M,\mu)
\leq \frac{\e}{2}$. 
We insist on the fact that the estimate in 
\eqref{eq:test_est_intro} holds 
{\it a priori} with respect to the choice of
a testing dataset and the 
corresponding computation of
the generalization error.}
We report that the minimization of the functional
\eqref{eq:G_lim_intro} plays a crucial role in
\cite{BCFH21}.

On the base of the previous theoretical results, 
we propose two possible training
algorithms to approximate the unknown
diffeomorphism $\Psi$, having observed its action
on an ensemble $\{ x^j_{0} \}_{j=1,\ldots,M}$.
The first one consists in introducing a 
finite-dimensional subspace $\U_N\subset \U$, 
and in projecting onto $\U_N$ the gradient flow
induced by $\F^M$ in the space of admissible 
controls. Here we make use of 
the gradient flow formulation
for optimal control problems derived in \cite{S21}. 
The second approach is based on the Pontryagin
Maximum Principle, and it is similar to one followed
in \cite{LCTE17}, \cite{BC19}, \cite{BCFH21} 
for control system \eqref{eq:ctrl_sys_ResNet}.
We used the iterative method proposed in
\cite{SS80} for the the numerical resolution
of Optimal Control problems.
The main advantage of dealing with a 
linear-control system is that the 
maximization of the Hamiltonian
(which is a key-step for the controls update)
can be carried out with a low computational effort.
In contrast, in \cite{LCTE17} a non-linear 
optimization software was employed
to manage this task, while in 
\cite{BCFH21} the authors wrote the maximization 
as a root-finding problem and solved it with 
Brent's method.\\
The paper is organized as follows.
In Section~\ref{sec:notations} we establish the
notations and we prove preliminary results regarding
the flow generated by control system 
\eqref{eq:lin_ctrl_sys_intro}.
In Section~\ref{sec:ens_ctrl} we recall 
some results contained in \cite{AS1} and \cite{AS2}
concerning exact and approximate controllability 
of ensembles.
In Section~\ref{sec:def_strat} we explain why 
the ``null training error strategy" is not suitable
for the approximation purpose, and we outline
the alternative strategy based on the resolution
of an Optimal Control problem.
In Section~\ref{sec:G-conv} we prove a 
$\Gamma$-convergence result that holds when the
size $M$ of the training dataset tends to infinity.
As a byproduct, we obtain the upper bound on the
expected generalization error reported in 
\eqref{eq:test_est_intro}. 
In Section~\ref{sec:proj_grad} we describe the 
training algorithm based on the projection of
the gradient flow induced by the cost $\F^M$.
In Section~\ref{sec:train_PMP} we propose the
training procedure based on the Pontryagin
Maximum Principle.
In Section~\ref{sec:num_exp} we test numerically 
the algorithms by approximating a diffeomorphism
in the plane.

\end{section}

\begin{section}{Notations and Preliminary Results}
\label{sec:notations}
In this paper we consider control systems
of the form
\begin{equation}\label{eq:subr_sys_notat}
\dot x =F(x)u= \sum_{i=1}^lF_i(x)u_i,
\end{equation}
where the controlled vector fields 
$(F_i)_{i=1,\ldots,l}$
satisfy the following assumption.

\begin{assumption} \label{ass:bound_fields}
The vector fields $F_1,\ldots,F_l$ are
smooth and bounded\footnote{With minor
adjustments, the hypothesis of 
boundedness of the vector fields can be slightly 
relaxed by requiring a sub-linear growth condition.
We preferred to prove the results for bounded 
vector fields in order to avoid technicalities.} in
$\R^n$, together with their derivatives of each order.
\end{assumption}

In particular, this implies that
there exists $C_1>0$ such that
\begin{equation}\label{eq:lipsc_fields}
\sup_{i=1,\ldots,l}\,\, \sup_{x,y\in\R^n} 
\frac{|F_i(x)-F_i(y)|_2}{|x-y|_2}\leq C_1
\end{equation} 
The space of admissible controls is 
$\U:= L^2([0,1],\R^l)$, endowed with the usual Hilbert
space structure. 
Using Assumption~\ref{ass:bound_fields},
the classical Carat\'eodory Theorem guarantees
that, for every $u\in\U$ and for every $x_0\in\R^n$,
the Cauchy problem
\begin{equation} \label{eq:Cauc_ctrl}
\begin{cases}
\dot x(s) = \sum_{i=1}^lF_i(x(s))u_i(s),\\
x(0)=x_0,
\end{cases}
\end{equation}
has a unique solution $x_{u,x_0}:[0,1]\to\R^n$.
Hence, for every $u\in\U$, we can define the flow
$\Phi_u:\R^n\to\R^n$ as follows:
\begin{equation} \label{eq:flow_def}
\Phi_u: x_0 \mapsto x_{u,x_0}(1),
\end{equation}
where  $x_{u,x_0}$ solves \eqref{eq:Cauc_ctrl}.
We recall that $\Phi_u$ is a diffeomorphism, since
it is smooth and invertible.

In the paper we will make extensive use of the 
weak topology of $\U$. We recall that a sequence
$(u_\ell)_{\ell\geq1}\subset\U$ is weakly convergent
to an element $u\in\U$ if, for every $v\in\U$,
we have that
\[
\lim_{\ell\to\infty}\int_0^1v(s)u_\ell(s)\,ds=
\int_0^1v(s)u(s)\,ds,
\]
and we write $u_\ell\weak_{L^2} u$.
In the following result we consider a sequence of
solutions of \eqref{eq:Cauc_ctrl}, corresponding
to a weakly convergent sequence of controls. 
\begin{lemma} \label{lem:conv_traj}
Let $(u_\ell)_{\ell\geq1}\subset\U$ be a weakly
convergent sequence, such that $u_\ell\weak_{L^2} u$
 as
$\ell\to\infty$.
For every $x_0\in\R^n$ and for every $\ell\geq1$, let
$x_{u_\ell,x_0}:[0,1]\to\R^n$ be the solution of
Cauchy problem \eqref{eq:Cauc_ctrl} corresponding
to the control $u_\ell$. Then
\begin{equation*}
\lim_{\ell\to\infty} 
||x_{u_\ell,x_0}-x_{u,x_0}||_{C^0}=0,
\end{equation*}
where $x_{u,x_0}:[0,1]\to\R^n$ is the solution of
Cauchy problem \eqref{eq:Cauc_ctrl} corresponding
to the control $u$.
\end{lemma}
\begin{proof}
Since the sequence $(u_\ell)_{\ell\geq1}$ is weakly
convergent, it is bounded in $\U$. This fact and
Assumption~\ref{ass:bound_fields} imply that
there exists a compact set $K\subset \R^n$ such that
$x_{u_\ell,x_0}(s)\in K$ for every $\ell\geq1$ and for
every $s\in[0,1]$. Moreover, being 
$(F_i)_{i=1,\ldots,l}$ bounded, we deduce that
\[
||\dot x_{u_\ell,x_0}||_{L^2} \leq C||u_l||_{L^2}
\]
for some constant $C>0$.
Then it follows that the sequence is 
bounded in $W^{1,2}([0,1],\R^n)$, 
and, as a matter of fact, that it is pre-compact
with respect to the weak topology of $W^{1,2}$.
Therefore, there exists a subsequence
$(x_{u_{\ell_k},x_0})_{k\geq 1}$ 
and $x_\infty\in W^{1,2}$ such that
$x_{u_{\ell_k},x_0}\weak_{W^{1,2}}x_\infty$
as $k\to\infty$.
Using the compact inclusion
$W^{1,2}([0,1],\R^n)\hookrightarrow C^0([0,1],\R^n)$,
we deduce that $x_{u_{\ell_k},x_0}\to_{C^0}x_\infty$.
In particular, we obtain that
$x_\infty(0) =x_0$.
On the other hand, we have that 
$F(x_{u_{\ell_k},x_0})u_{\ell_k} \weak_{L^2} F(x_\infty)u$, since $F(x_{u_{\ell_k},x_0})
\to_{C^0} F(x_\infty)$. Recalling  that 
$\dot x_{u_{\ell_k}}= 
F(x_{u_{\ell_k},x_0})u_{\ell_k}$
and that 
$x_{u_{\ell_k}} \weak_{W^{1,2}} \dot x_{\infty}$
implies
$\dot x_{u_{\ell_k}} \weak_{L^2} \dot x_{\infty}$,
we obtain that 
\[
\begin{cases}
\dot x_\infty = F(x_\infty)u,\\
x_\infty(0) = x_0,
\end{cases}
\]
i.e., $x_\infty \equiv x_{u,x_0}$. Since the reasoning
above is valid for every possible $W^{1,2}$-weak
limiting point of $(x_{u_{\ell},x_0})_{\ell\geq1}$,
we deduce that the whole sequence is 
$W^{1,2}$-weakly convergent
to $x_{u,x_0}$. Using again the compact inclusion
$W^{1,2}([0,1],\R^n)\hookrightarrow C^0([0,1],\R^n)$,
we deduce the thesis.
\end{proof}

We now prove an estimate of the Lipschitz constant
of the flow $\Phi_u:\R^n\to\R^n$ for $u\in\U$.

\begin{lemma} \label{lem:lipsch_flows}
For every admissible control
$u\in\U$, let $\Phi_u:\R^n\to\R^n$ be
corresponding flow defined in \eqref{eq:flow_def}.
Then $\Phi_u$ is Lipschitz-continuous with constant
\begin{equation} \label{eq:lipsc_flow}
L_{\Phi_u}\leq e^{C||u||_{L^2}},
\end{equation}
where $C>0$ depends only on the controlled 
fields $(F_i)_{i=1,\ldots,l}$ and on the dimension
of the ambient space.   
\end{lemma}
\begin{proof}
For every $x_0,w\in\R^n$ we have that 
$D_{x_0}\Phi_u(w)=v(1)$, where $v:[0,1]\to\R^n$
solves the linearized system
\[
\begin{cases}
\dot v(s) = \left(\sum_{i=1}^l u_i(s)
D_xF_i(x_{u,x_0}(s))  \right)v(s),\\
v(0)=w,
\end{cases}
\]
where $x_{u,x_0}:[0,1]\to\R^n$ is the controlled
trajectory that solves \eqref{eq:Cauc_ctrl}, and 
for every $i=1,\ldots,l$ the expression
$D_xF_i=(D_{x_1}F_i,\ldots,D_{x_n}F_i)$ 
denotes the Jacobian of the field $F_i$. 
Therefore, using \eqref{eq:lipsc_fields}, we
deduce that 
\[
\frac{d}{ds}|v(s)|^2_2 \leq 2
C_1
\left( \sum_{i=1}^l|u_i(s)|
\right)
|v(s)|^2_2.
\]
Recalling that $\sum_{i=1}^l|u_i(s)|\leq 
\sqrt n |u(s)|_2$ and that
\[
\int_0^1|u(s)|_2 \, ds 
\leq \left( \int_0^1|u(s)|_2^2 \, ds \right)^{\frac12}
= ||u||_{L^2},
\]
we deduce that
\[
|v(1)|_2^2 \leq e^{2C||u||_{L^2}} |v(0)|_2^2,
\]
where $C>0$ is a constant that depends only on 
$C_1$ and on the dimension $n$ of the ambient space.
Finally this shows that
\[
\frac{|D_{x_0}\Phi_u(w)|_2}{|w|_2}\leq 
e^{C||u||_{L^2}},
\]
proving the thesis.
\end{proof}

The next result regards the convergence of the 
flows $(\Phi_{u_\ell})_{\ell\geq1}$ corresponding to
a weakly convergent sequence 
$(u_\ell)_{\ell\geq1} \subset\U$.

\begin{proposition}\label{prop:conv_flows}
Given a sequence $(u_\ell)_{\ell\geq 1}\subset \U$ 
such that $u_\ell\weak_{L^2} u$ as $\ell\to\infty$
with respect to the weak topology of $\U$, then
the flows $(\Phi_{u_\ell})_{\ell\geq1}$ converge to 
$\Phi_u$ uniformly on compact sets.
\end{proposition}
\begin{proof}
Let $K\subset\R^n$ be a compact set. 
For every $\ell\geq1$ let us define 
$\Phi_{u_\ell,K}:=\Phi_{u_\ell}|_K$,
and $\Phi_{u,K}:=\Phi_{u}|_K$.
Being weakly convergent, the sequence 
$(u_\ell)_{\ell\geq1}$ is bounded in $\U$. This fact
and Assumption~\ref{ass:bound_fields}
imply
 that there exists a compact set $K'\supset K$
such that any solution of \eqref{eq:Cauc_ctrl}
with initial datum in $K$ is entirely contained in
$K'$. In particular, we deduce that
$\Phi_{u_\ell}(K)\subset K'$ for every $\ell\geq1$,
i.e., the sequence $(\Phi_{u_\ell,K})_{\ell\geq1}$
is equi-bounded.
Moreover, using the boundedness 
of $(u_\ell)_{\ell\geq1}\subset\U$ 
and Lemma~\ref{lem:lipsch_flows}, we obtain that
$(\Phi_{u_\ell,K})_{\ell\geq1}$
is an equi-continuous family of diffeomorphisms.
By Ascoli-Arzel\`a Theorem, it follows
that $(\Phi_{u_\ell,K})_{\ell\geq1}$ is pre-compact
with respect to the $C^0$ topology.
On the other hand, Lemma~\ref{lem:conv_traj} 
implies that for every $x\in K$
$\Phi_{u_\ell,K}(x)\to \Phi_{u,K}(x)$
as $\ell\to\infty$.
This guarantees that the set of limiting points
with respect to the $C^0$ topology
of the sequence  $(\Phi_{u_\ell,K})_{\ell\geq1}$
is reduced to the single element $\Phi_{u,K}$.
\end{proof}

\end{section}

\begin{section}{Ensemble Controllability}\label{sec:ens_ctrl}
In this section we recall the most important results
regarding the issue of ensemble controllability.
For the proofs and the statements in full generality
we refer the reader to \cite{AS1} and \cite{AS2}.
We begin with the definition of ensemble in $\R^n$.
In this section we will further assume that
$n>1$, which is the most interesting case.

\begin{definition}\label{def:ensemble}
Given a compact set $\Theta\subset \R^n$, an 
{\it ensemble of points}
 in $\R^n$ is an injective and continuous map
$\gamma:\Theta\to\R^n$. 
We denote by $\ensemb(\R^n)$ the space of 
ensembles.    
\end{definition}
\begin{remark}
If $|\Theta|=M<\infty$, then an ensemble 
can be simply thought
as an injective map from $\{1,\ldots,M \}$
to $\R^n$, or, equivalently,
as an element of $(\R^n)^M\setminus \Delta^{(M)}$,
where 
\[\Delta^{(M)}:=\{ (x^1,\ldots,x^M)
\in (\R^n)^M: \exists j_1\neq j_2
\mbox{ s.t. } x^{j_1}=x^{j_2} \}.
\]
We define $(\R^n)^{(M)}:=(\R^n)^M\setminus 
\Delta^{(M)}$. Given a vector field $F:\R^n\to\R^n$, we define
its {\it $M$-fold} vector field 
$F^{(M)}:(\R^n)^{(M)}\to(\R^n)^{(M)}$ as
$F^{(M)}(x^1,\ldots,x^M)=(F(x^1),\ldots,F(x^M))$,
for every $(x^1,\ldots,x^M)\in (\R^n)^{(M)}$.
Finally, we introduce the notation $\ensembM(\R^n)$
to denote the space of ensembles of $\R^n$ with
$M$ elements.
\end{remark}

We give the notion of {\it reachable} ensemble.
\begin{definition}\label{def:ex_reach}
The ensemble $\gamma(\cdot)\in \ensemb(\R^n)$
is {\it reachable} from the ensemble $\alpha(\cdot)\in 
\ensemb(\R^n)$ if there exists an admissible
control $u\in\U$ such that its
corresponding flow $\Phi_u$ defined in 
\eqref{eq:flow_def} satisfies:
\[
\Phi_u(\alpha(\cdot)) = \gamma(\cdot).
\]
We can equivalently say that 
$\alpha(\cdot)$ can be steered to 
$\gamma(\cdot)$.
\end{definition}

\begin{definition}\label{def:ex_controll}
Control system \eqref{eq:subr_sys_notat} is
{\it exactly controllable} from $\alpha(\cdot)\in 
\ensemb(\R^n)$ if every $\gamma(\cdot)
\in \ensemb(\R^n)$ is reachable from $\alpha(\cdot)$.
The system is {\it globally exactly controllable}
if it is exactly controllable from every
$\alpha(\cdot)\in \ensemb(\R^n)$.
\end{definition}

We recall the definition of Lie algebra generated by
a system of vector fields.
Given the vector fields $F_1,\ldots,F_l$,
the linear space $\mathrm{Lie}(F_1,\ldots,F_l)$ is
defined as
\begin{equation} \label{eq:Lie_gen}
\mathrm{Lie}(F_1,\ldots,F_l):=\mathrm{span}
\{ [F_{i_s},[\ldots ,[F_{i_2},F_{i_1}]\ldots]]
: s\geq1, i_1,\ldots,i_s\in \{ 1,\ldots,l \}\},
\end{equation}
where $[F,F']$ denotes the Lie bracket between 
$F,F'$, smooth vector fields of $\R^n$.
In the case of finite ensembles, i.e., when 
$|\Theta|=M<\infty$,
we can provide sufficient condition
for controllability. The proof reduces to the
classical Chow-Rashevsky theorem (see ,e.g, the
textbook \cite{AS}).

\begin{theorem}\label{thm:fin_controll}
Let $F_1,\ldots,F_l$ be a system of vector fields
on $\R^n$.
Given $M\geq 1$, if for
every $x^{(M)}=(x^1,\ldots,x^M)\in (\R^n)^{(M)}$
the system of 
M-fold vector fields $F^{(M)}_1,\ldots,F^{(M)}_l$ 
is bracket generating at $x^{(M)}$, i.e.,
\begin{equation} \label{eq:brack_gen}
\mathrm{Lie}(F^{(M)}_1,\ldots,F^{(M)}_l)|_{x^{(M)}}
= (\R^n)^M,
\end{equation}
 then
the control system \eqref{eq:subr_sys_notat}
is globally exactly controllable on
$\ensembM(\R^n)$.
\end{theorem}

For a finite ensemble, the global exact 
controllability holds for a generic system of
vector fields. 

\begin{proposition} \label{prop:gener_finite}
For every $l\geq2$, $M\geq 1$ and 
$m$ sufficiently large,
then the $l$-tuples of vector fields $(F_1,\ldots,F_l)
\in (\mathrm{Vect}(\R^n))^l$ such that
system \eqref{eq:subr_sys_notat} is globally exactly
controllable on $\ensembM(\R^n)$ is residual
with respect to the Whitney $C^m$-topology.
\end{proposition}

We recall that a set is said {\it residual} if
it is the complement of a countable union of
nowhere dense sets. 
Proposition~\ref{prop:gener_finite} 
means that, given any $l$-tuple $(F_1,\ldots,F_l)$ of
vector fields, the corresponding control system
\eqref{eq:subr_sys_notat} can be made 
globally exactly controllable in $\ensembM(\R^n)$
 by means of an arbitrary small perturbation 
of the fields $F_1,\ldots,F_l$
in the $C^m$-topology.

When dealing with infinite ensembles, the 
notions of ``exact reachable" and 
``exact controllable" are too strong.
However, they can be replaced by their
respective $C^0$-approximate versions.

\begin{definition} \label{def:app_reach}
The ensemble $\gamma(\cdot)\in \ensemb(\R^n)$
is {\it $C^0$-approximately
reachable} from the ensemble $\alpha(\cdot)\in 
\ensemb(\R^n)$ if for every
$\e>0$ there exists an admissible
control $u\in\U$ such that its
corresponding flow $\Phi_u$ defined in 
\eqref{eq:flow_def} satisfies:
\begin{equation} \label{eq:app_reach}
\sup_{\Theta}|\Phi_u(\alpha(\cdot)) - \gamma(\cdot)
|_2
<\e.
\end{equation}
We can equivalently say that 
$\alpha(\cdot)$ can be $C^0$-approximately steered to 
$\gamma(\cdot)$.
\end{definition}

\begin{definition}\label{def:app_controll}
Control system \eqref{eq:subr_sys_notat} is
{\it $C^0$-approximately controllable} from 
$\alpha(\cdot)\in \ensemb(\R^n)$
 if every $\gamma(\cdot)
\in \ensemb(\R^n)$ is 
$C^0$-approximately reachable from $\alpha(\cdot)$.
The system is {\it globally
$C^0$-approximately  controllable}
if it is $C^0$-approximately controllable from every
$\alpha(\cdot)\in \ensemb(\R^n)$.
\end{definition}

\begin{remark}
Let us further assume that 
the compact set $\Theta \subset \R^n$
has positive Lebesgue measure, and that it is
equipped with a finite and positive measure $\mu$,
absolutely continuous w.r.t. the Lebesgue measure.
Then, the distance between the target ensemble 
$\gamma(\cdot)$ and the approximating ensemble
$\Phi_u(\alpha(\cdot))$ can be quantified using
the $L_\mu^p$-norm:
\[
||\Phi_u(\alpha(\cdot)) -\gamma(\cdot)||_{L_\mu^p}=
\left(
\int_\Theta |\Phi_u(\alpha(\theta)) -\gamma(\theta)|_2^p
\,d\mu(\theta)
\right)^{\frac1p},
\]
and we can equivalently formulate the notion
of {$L_\mu^p$-approximate} controllability.
In general, given a non-negative continuous function
$a:\R^n\to\R$ such that $a(0)=0$, we can express
the approximation error as
\begin{equation}\label{eq:int_cost}
\int_\Theta a(\Phi_u(\alpha(\theta)) -\gamma(\theta))
\,d\mu(\theta).
\end{equation}
In Section~\ref{sec:G-conv}
we will consider an integral
penalization term of this form.
It is important to observe that, if $\gamma(\cdot)$ is 
$C^0$-approximately reachable from $\alpha(\cdot)$,
then \eqref{eq:int_cost} can be made arbitrarily
small.
\end{remark}

Before stating the next result we introduce some 
notations. 
Given a vector field $X:\R^n\to\R^n$ and a compact set 
$K\subset \R^n$, we define
\[
||X||_{1,K}:= \sup_{x\in K}\left(|X(x)|_2 + 
\sum_{i=1}^n|D_{x_i}X(x)|_2 \right). 
\]
Then we set
\[
\mathrm{Lie}^\delta_{1,K}(F_1,\ldots,F_l):=
\{
X\in \mathrm{Lie}(F_1,\ldots,F_l):
||X||_{1,K}\leq \delta
\}.
\]
We now formulate the assumption required
for the approximability result.
\begin{assumption} \label{ass:str_Lie}
The system of vector fields $F_1,\ldots,F_l$ satisfies
the {\it Lie algebra strong approximating property},
i.e., there exists $m\geq1$ such that, for every
$C^m$-regular vector field $Y:\R^n\to\R^n$ and for
each compact set $K\subset\R^n$ there exists 
$\delta>0$ such that
\begin{equation} \label{eq:Lie_alg_cond}
\inf \left\{ \sup_{x\in K}|X(x)-Y(x)|_2 \,\, \Big|
X\in \mathrm{Lie}_{1,K}^\delta (F_1,\ldots,F_l) \right
\}=0.
\end{equation}
\end{assumption}

The next result establishes a Universal Approximating
Property for flows.

\begin{theorem}\label{thm:univ_app_flows}
Let $\Psi:\R^n\to\R^n$ be a diffeomorphism diffeotopic
to the identity. Let $F_1,\ldots,F_l$ be a system of 
vector fields satisfying 
Assumptions~\ref{ass:bound_fields} 
and~\ref{ass:str_Lie}. Then for each compact set
$K\subset\R^n$ and each $\e>0$ there exists an
admissible control $u\in\U$ such that 
\begin{equation}\label{eq:Univ_app}
\sup_{x\in K}|\Psi(x)-\Phi_u(x)|_2\leq \e,
\end{equation}
where $\Phi_u$ is the flow corresponding to
the control $u$ defined in \eqref{eq:flow_def}.
\end{theorem}
We recall that $\Psi:\R^n\to\R^n$ is diffeotopic to the 
identity if and only if
there exists a family of diffeomorphisms
$(\Psi_s)_{s\in[0,1]}$ smoothly depending on $s$
 such that $\Psi_0=\mathrm{Id}$ and
$\Psi_1=\Psi$. In this case, $\Psi$ can be seen
as the flow generated by the non-autonomous vector
field $(s,x)\mapsto Y_s(x)$, where
\[
Y_s(x) := \left.
\frac{d}{d\e}\right|_{\e=0}\Psi_{s+\e}(x).
\]
From Theorem~\ref{thm:univ_app_flows} we can deduce
a $C^0$-approximate reachability result for infinite
ensembles.

\begin{corollary}
Let $\alpha(\cdot),\beta(\cdot)\in \ensemb(\R^n)$
be diffeotopic, i.e., there exists a diffeomorphism
$\Psi:\R^n \to \R^n$ diffeotopic
to the identity such that 
$\gamma = \Psi \circ \alpha$. Then $\gamma(\cdot)$
is $C^0$-approximate reachable from $\alpha(\cdot)$.
\end{corollary}

\begin{remark} \label{rem:str_Lie_br_gen}
If a system of vector fields satisfies
Assumption~\ref{ass:str_Lie}, then, for every
$M\geq1$ and for every $x^{(M)}\in(\R^n)^{(M)}$
Lie bracket generating condition \eqref{eq:brack_gen}
is automatically satisfied 
(see \cite[Theorem~4.3]{AS2}). This means that
Assumption~\ref{ass:str_Lie} guarantees 
global exact controllability in
$\ensembM(\R^n)$,
and $C^0$-approximate reachability 
for infinite diffeotopic ensembles.
\end{remark}

We conclude this section with the exhibition of 
a system of vector fields in $\R^n$ meeting
Assumptions~\ref{ass:bound_fields} 
and~\ref{ass:str_Lie}.

\begin{theorem} \label{thm:sys_contr}
For every $n>1$ and $\nu>0$,
consider the vector fields in $\R^n$
\begin{equation}\label{eq:v_fields_control}
{\bar F_i}(x) := \frac{\partial}{\partial x_i},
\quad 
{\bar F'_i}(x) := e^{-\frac{1}{2\nu}|x|^2}\frac{\partial}{\partial x_i},
\quad i=1,\ldots,n.
\end{equation}
Then the system {
$\bar F_1,\ldots,\bar F_n,\bar F'_1,\ldots,\bar F'_n$}
satisfies Assumptions~\ref{ass:bound_fields} 
and~\ref{ass:str_Lie}. 
\end{theorem} 

\begin{remark} \label{rmk:sys_fields}
If we consider the vector fields 
$\bar F_1,\ldots, \bar F_n,
\bar F'_1,\ldots,\bar F'_n$
defined in \eqref{eq:v_fields_control},
then Theorem~\ref{thm:univ_app_flows} 
and Theorem~\ref{thm:sys_contr} guarantee 
that the flows generated by the 
corresponding linear-control system can
approximate on compact sets diffeorphisms
diffeotopic to the identity.
From a theoretical viewpoint, this approximation
result cannot be strengthened by enlarging the 
family of controlled vector fields, since the flows
produced by any controlled dynamical system are
themselves diffeotopic to the identity.
On the other hand, in view of the discretization
of the dynamics and the consequent construction of   
the ResNet, it could be useful to enrich the
system of the controlled fields.
As suggested by Assumption~\ref{ass:str_Lie}, 
the expressivity of the linear-control 
system is more directly related to 
the space $\mathrm{Lie}(F_1,\ldots,F_l)$,
rather than to the family $F_1,\ldots,F_l$ 
itself. However, as we are going to see, 
``reproducing''  the flow of a field that
belongs to $\mathrm{Lie}(F_1,\ldots,F_l)
\setminus \mathrm{span}\{ F_1,\ldots,F_l \}$
can be expensive.
Let us consider an evolution step-size $h\in(0,1/4)$
and let us choose two of the controlled
vector fields, say $F_1,F_2$, and let us assume that
$[F_1,F_2]\in \mathrm{Lie}(F_1,\ldots,F_l)
\setminus \mathrm{span}\{ F_1,\ldots,F_l \}$.
Let us denote by 
$e^{\pm hF_i}:\R^n\to\R^n, i=1,2$ the flows 
obtained by evolving $\pm F_i, i=1,2$ for 
an amount of time equal to $h$.
Then, using for instance the computations 
in \cite[Subsection~2.4.7]{AS},
for every $x\in K$ compact we obtain that
\begin{equation*}
\left( e^{-hF_2}\circ e^{-hF_1}\circ
e^{hF_2}\circ e^{hF_1}\right) (x)
=  e^{h^2[F_1,F_2]}(x) + o(h^2)
\end{equation*}
as $h\to0$.
The previous computation shows that, in order
to approximate the effect of evolving the vector
field $[F_1,F_2]$ for an amount of time equal to 
$h^2$, we need tho evolve the fields $\pm F_i, i=1,2$
for a total amount $4h$. If $h$ represents the
discretization step-size used to derive the
ResNet \eqref{eq:ResNet_lin_intro}
from the linear-control system
\eqref{eq:subr_sys_notat} on the interval $[0,1]$,
 then we have that
$h=\frac1N$, where $N$ is the number of layers of the
ResNet. The argument above suggests that we need 
to use $4$ layers of the network to ``replicate''
the effect of evolving $[F_1,F_2]$ for the amount
of time $h^2=\frac1{N^2}$ (note that $h^2\ll h$ 
when $N\gg 1$).
This observation provides an insight 
for the practical choice of the system of 
controlled fields. In first place, the system
$F_1,\ldots,F_l$ should meet 
Assumption~\ref{ass:str_Lie}. 
If the ResNet obtained from the
discretization of the system does not seem to be  
expressive enough, it should be considered to
enlarge the family of the controlled fields,
for example
by including some elements of
$\mathrm{span}\{[F_{i_1},F_{i_2}]:i_1,i_2\in \{1,
\ldots,l \}\}$ (or, more generally, of
$\mathrm{Lie}(F_1,\ldots,F_l)$). 
We insist on the fact that
this procedure increases the width of the network,
since the larger is the number of 
 fields in the control system,
the larger is the number of
parameters per layer in the ResNet.
\end{remark}

\end{section}

\begin{section}{Approximation of Diffeomorphisms:
Robust Strategy} \label{sec:def_strat}
In this section we introduce the central problem
of the paper, i.e.,
the {\it training} of  control system 
\eqref{eq:subr_sys_notat}
in order to approximate
an unknown diffeomorphism $\Psi:\R^n\to\R^n$
diffeotopic to the identity.
A typical situation that may arise in applications
is that we want to approximate $\Psi$ on a compact
set $K\subset\R^n$, having observed the action of
$\Psi$ on a finite number of {\it training} points 
$\{ x^1_{0},\ldots,x^M_{0} \}\subset K$.
Our aim is to formulate a strategy that is 
{\it robust} with respect to the size $M$ of the
training dataset, and for which we can 
give upper bounds for the {\it generalization} error.   
In order to obtain higher and higher
degree of approximation,
we may thing to triangulate the compact set
$K$ with an increasing number of nodes where we
can evaluate the unknown map $\Psi$.
Using the language introduced in 
Section~\ref{sec:ens_ctrl}, we have that, for every
$M\geq1$, we may understand
a triangulation of $K$ with $M$ nodes
as an ensemble $\alpha^M(\cdot)\in\ensembM(\R^n)$.
After evaluating $\Psi$ in the nodes, we obtain
the target ensemble
$\gamma^M(\cdot)\in \ensembM(\R^n)$ as 
$\gamma^M(\cdot):=\Psi(\alpha^M(\cdot))$.

If the vector fields $F_1,\ldots,F_l$ that
define control system \eqref{eq:subr_sys_notat}
meet Assumptions~\ref{ass:bound_fields} 
and~\ref{ass:str_Lie}, then 
Theorem~\ref{thm:fin_controll}
and Remark~\ref{rem:str_Lie_br_gen} may suggest
a first natural attempt to design an approximation
strategy. Indeed, for every $M$, we can
{\it exactly} steer the ensemble $\alpha^M(\cdot)$
to the ensemble $\gamma^M(\cdot)$ with an admissible
control $u^M\in\U$.
Hence, we can choose the flow $\Phi_{u^M}$ defined
in \eqref{eq:flow_def}
as an approximation of $\Psi$ on $K$, achieving 
a {\it null training error}.
Assume that, for every $M\geq1$, the
corresponding triangulation is a 
$\epsilon^M$-approximation
of the set $K$, i.e., for every $y\in K$ there
exists $\bar j\in\{ 1,\ldots,M \}$ such that 
$|y-x^{\bar j, M}_{0}|_2\leq \epsilon^M$, where 
$x^{\bar j, M}_{0}:=\alpha^M(\bar j)$.
Then, for every $y\in K$,
we can give the following estimate for the
generalization error:
\begin{align*}
|\Psi(y) - \Phi_{u^M}(y)|_2&\leq
|\Psi(y) - \Psi(x^{\bar j, M}_{0})|_2 +
|\Phi_{u^M}(x^{\bar j, M}_{0}) - \Phi_{u^M}(y)|_2 \\
&\leq L_{\Psi}\epsilon^M + L_{\Phi_{u^M}}\epsilon^M,
\end{align*}
where $L_{\Psi},  L_{\Phi_{u^M}}$ are
respectively the Lipschitz constants of 
$\Psi$ and $\Phi_{u^M}$.
Assuming (as it is natural to do)
that $\epsilon^M \to 0$ as $M\to\infty$, 
the strategy of achieving zero training error
works if, for example, the Lipschitz constants
of the approximating flows $(\Phi_{u^M})_{M\geq1}$
keep bounded. This in turn would follow if the
sequence of controls $(u^M)_{M\geq 1}$ were
bounded in $L^2$-norm. However, as we are going
to see in the following proposition, in general
this is not the case.
Let us define 
\begin{equation*}
\mathrm{Flows}_K(F_1,\ldots,F_l):=
\{
\Phi_u: u\in \U
\},
\end{equation*}
the space of flows restricted to $K$
obtained via \eqref{eq:flow_def}
with admissible controls, and let 
$\mathrm{Diff}_K^0(\R^n)$ be the space of 
diffeomorphisms diffeotopic to the identity
restricted to $K$.
Theorem~\ref{thm:univ_app_flows} guarantees that,
for every $K\subset \R^n$,
\[
\overline{\mathrm{Flows}_K(F_1,\ldots,F_l)}^{C^0}
= \mathrm{Diff}_K^0(\R^n).
\]
\begin{proposition} \label{prop:unbound_app}
Given a diffeomorphism
 $\Psi \in \mathrm{Diff}_K^0(\R^n)
\setminus \mathrm{Flows}_K(F_1,\ldots,F_l)$ and
an approximating sequence 
$(\Phi_{u_\ell})_{\ell\geq1} \subset 
\mathrm{Flows}_K(F_1,\ldots,F_l)$ such that
$\Phi_{u_\ell}\to_{C^0}\Psi$ on $K$, then
the sequence of controls $(u_\ell)_{\ell\geq1}\subset
\U$ is unbounded in the $L^2$-norm.
\end{proposition}
\begin{proof}
By contradiction, let $(u_\ell)_{\ell\geq1}$
be a bounded sequence in $\U$. Then, we can extract
a subsequence $(u_{\ell_k})_{k\geq1}$ weakly
convergent to $u\in\U$. In virtue of
Proposition~\ref{prop:conv_flows}, we have that
$\Phi_{u_{\ell_k}}\to_{C^0}\Phi_u$ on $K$. However,
since $\Phi_{u_\ell}\to_{C^0}\Psi$ on $K$, we deduce
that $\Psi =\Phi_u$ on $K$, but this contradicts the
hypothesis $\Psi \in \mathrm{Diff}_K^0(\R^n)
\setminus \mathrm{Flows}_K(F_1,\ldots,F_l)$.
\end{proof}

The previous result sheds light on a
weakness of the approximation strategy described 
above. Indeed, the main drawback is that we have no
bounds on the norm of the controls $(u^M)_{M\geq1}$,
and therefore, even though the triangulation 
of $K$ is fine, we cannot give an {\it a priori}
estimate of the testing error. 
We point out that, in 
the different framework of simultaneous
control of several systems, a similar situation
was described in \cite{ABS}. 

\subsection{Approximation via Optimal Control}
In order to avoid the issues described above,
we propose a training
strategy based on the solution of an
Optimal Control problem with a regularization term
penalizing the $L^2$-norm of the control.
Namely, given a set of training points
$\{ x^1_{0},\ldots,x^M_{0} \}\subset K$,
 we consider the nonlinear functional
$\F^M:\U\to\R$ defined as follows:
\begin{equation}\label{eq:fun_fin}
\F^M(u):= \frac1M \sum_{j=1}^M 
a(\Phi_u(x^j_{0})-\Psi(x^j_{0})) + \frac\beta2 
|| u ||^2_{L^2},
\end{equation}
where $a:\R^n\to\R$ is a smooth loss
function such that
$a\geq 0$ and $a(0) =0$, and $\beta>0$ is a fixed 
parameter. 
The functional $\F^M$ is composed by two competing
terms: the first aims at achieving a low
{\it mean} training error, while the second aims at
keeping bounded the $L^2$-norm of the control.
In this framework, it is worth assuming that the 
compact set $K$ is equipped with a Borel probability 
measure $\mu$. In this way, we can give higher
weight to the regions in $K$ where we need more
accuracy in the approximation. 
As done before, for every $M\geq1$
 we understand the training
dataset as an ensemble $\alpha^M(\cdot)\in
\ensembM(\R^n)$. Moreover, we associate to it the 
discrete probability measure $\mu_M$ defined as 
\begin{equation}\label{eq:discr_meas}
\mu_M := \frac1M \sum_{j=1}^M\delta_{\alpha(j)},
\end{equation}
and we can equivalently express the mean training
error as
\begin{equation*}
\frac1M \sum_{j=1}^M 
a(\Phi_u(x^j_{0})-\Psi(x^j_{0})) = \int_{\R^n}a(\Phi_u(x)-
\Psi(x))\, d\mu_M(x).
\end{equation*}
From now on, when considering datasets growing in 
size, we make the following assumption on the sequence
of probability measures $(\mu_M)_{M\geq1}$.
\begin{assumption} \label{ass:weak_conv}
There exists a Borel probability measure 
$\mu$ supported in the compact set $K\subset \R^n$
such that the sequence
$(\mu_M)_{M\geq 1}$ is weakly convergent to $\mu$, 
i.e.,
\begin{equation} \label{eq:weak_conv_meas}
\lim_{M\to\infty} \int_{\R^n} f(x)\, d\mu_M =
\int_{\R^n} f(x)\, d\mu,
\end{equation}
for every bounded continuous function $f:\R^n\to\R$.
Moreover, we ask that $\mu_M$ 
is supported in $K$ for every $M\geq1$.  
\end{assumption}  
\begin{remark}
The request of Assumption~\ref{ass:weak_conv} is
rather natural. Indeed, if the elements of the
ensembles $\alpha^M(\cdot)\in \ensembM(\R^n)$ are 
sampled using the probability
distribution $\mu$ associated to the compact set $K$, 
it follows from the law of large numbers that
\eqref{eq:weak_conv_meas} holds.
On the other hand, since we ask that 
all the ensembles are contained in the
compact set $K$,
we have that the sequence of probability
measures $(\mu_M)_{M\geq 1}$  is tight. Therefore,
in virtue of Prokhorov Theorem, $(\mu_M)_{M\geq 1}$
is sequentially weakly pre-compact (for details
see, e.g., \cite{C10}).
\end{remark}

\begin{remark}
When $K=\overline{\mathrm{int}(K)}$, if 
for every $M$ $\alpha^M(\cdot)$ is a 
$\epsilon^M$-approximation of $K$
such that $\epsilon^M\to 0$ as $M\to\infty$,
then the corresponding sequence of probability
measures $(\mu_M)_{M\geq1}$ is weakly convergent to
$\mu = \mathcal{L}|_K$, where $\mathcal{L}$ denotes
the Lebesgue measure in $\R^n$.
\end{remark}

In the next result we prove that the functional
$\F^M$ attains minimum for every $M\geq1$.
Moreover, we can give an upper bound for the
$L^2$-norm of the minimizers of
$\F^M$ that does not depend
on $M$.

\begin{proposition}\label{prop:bound_norm}
For every $M\geq 1$ the functional $\F^M:\U\to\R$
defined in \eqref{eq:fun_M} admits a minimizer.
Moreover, if Assumption~\ref{ass:weak_conv} is
met, then there exists $C_\beta>0$ such that, for 
every $M\geq 1$, any minimizer $\tilde u^M$ of $\F^M$
satisfies the following inequality:
\begin{equation} \label{eq:up_bound_norm}
||\tilde u^M||_{L^2} \leq C_\beta.
\end{equation}
\end{proposition}
\begin{proof}
The existence of minimizers
descends from the direct method of Calculus of 
Variations, since  for every $M\geq 1$ the functional
$\F^M$ is coercive and lower semi-continuous
with respect to the weak topology of $\U$.
Indeed, recalling that $a\geq 0$, we have that
\[
\{ u\in \U: \F^M(u)\leq c \} \subset 
\{
u\in \U: \beta ||u ||_{L^2}^2 \leq 2c
\},
\]
and this shows that $\F^M$ is coercive, for every
$M\geq 1$. As regards the lower semi-continuity, 
let us consider a sequence 
$(u_\ell)_{\ell\geq1} \subset \U$ such that
$u_\ell\weak u$ as $\ell\to\infty$. 
Using Proposition~\ref{prop:conv_flows}
and the Dominated Convergence Theorem,
for every $M\geq1$ we have that
\begin{equation*}
\lim_{\ell\to\infty} \int_{\R^n}a(\Phi_{u_\ell}(x)
-\Psi(x))\,d\mu_M(x) =
\int_{\R^n}a(\Phi_{u}(x)
-\Psi(x))\,d\mu_M(x).
\end{equation*}  
Finally, recalling that 
\[
||u||_{L^2}^2 \leq 
\liminf_{\ell\to\infty}||u_\ell||_{L^2}^2,
\]
we deduce that for every $M\geq1$ the functional
$\F^M$ is lower semi-continuous. 

We now prove the uniform estimate on the 
norm of minimizers.
For every $M\geq 1$, let $\tilde u^M$ be a minimizer
of $\F^M$.
If we consider $v\equiv 0$, recalling that 
$\Phi_v = \mathrm{Id}$, we have that
\[
\F^M(\tilde u^M) \leq \F^M(v) 
=\int_{\R^n}a(x-\Psi(x))\,d\mu_M.
\]
Using the fact that $(\mu_M)_{M\geq1}$ are uniformly
compactly supported, we deduce that there exists
$C>0$ such that
\[
\sup_{M\geq1} \int_{\R^n}a(x-\Psi(x))\,d\mu_M
\leq C.
\]
Therefore, recalling that $a\geq0$, we have that
\[
\frac\beta2 ||\tilde u^M||^2_{L^2} \leq C,
\]
for every $M\geq1$. This concludes the proof.
\end{proof}

The previous result suggests as a training strategy
to look for a minimizer of the functional $\F^M$.
In the next section we investigate 
the properties of the functionals $(\F^M)_{M\geq1}$
using the the tools of $\Gamma$-convergence.

\end{section}

\begin{section}{Ensembles growing in size and $\Gamma$-convergence}\label{sec:G-conv}
In this section we study the limiting problem 
when the size of the training dataset tends to
infinity. The main fact is that a $\Gamma$-convergence
result holds.  Roughly speaking, this means that
{\it increasing the size
 of the training dataset does not
make the problem harder, at least from a theoretical
viewpoint}.
For a thorough introduction to the subject, 
the reader is referred to \cite{D93}.

For every $M\geq 1$,
let $\alpha^M(\cdot)\in \ensembM(\R^n)$ be an
ensemble of points in the compact set
$K\subset\R^n$, and let us consider the 
discrete probability measure $\mu_M$ defined 
in \eqref{eq:discr_meas}.
For every $M\geq1$ we consider the functional
$\F^M:\U\to\R$ defined as follows:
\begin{equation} \label{eq:fun_M}
\F^M(u) := \int_{\R^n}a(\Phi_u(x)-\Psi(x))\,d\mu_M + \frac{\beta}{2}||u||^2_{L^2}.
\end{equation}
The tools of $\Gamma$-convergence requires
the domain where the functionals are defined
to be equipped with a metrizable topology.
Recalling that the weak topology of $L^2$
is metrizable only on bounded sets, 
we need to properly restrict the functionals.
For every $\rho>0$, we set
\[
\U_\rho := \{ u\in\U: ||u||_{L^2}\leq \rho \}.
\]
Using Proposition~\ref{prop:bound_norm} we can choose
$\rho=C_\beta$, so that
\[
\arg \min_{\U} \F^M = \arg \min_{\U_\rho} \F^M,
\]
for every $M\geq 1$. With this choice we restrict 
the minimization problem to a bounded subset of $\U$, 
without losing any minimizer. We define 
$\F^M_\rho:=\F^M|_{\U_\rho}$.
We recall the definition of $\Gamma$-convergence.
\begin{defn*}
The family of functionals
$(\F^M_\rho)_{M\geq 1}$ is said to
 $\Gamma$-converge
to a functional 
$\F_\rho:\U_\rho\to\R\cup\{ +\infty \}$
with respect to the weak topology of $\U$ as 
$M\to \infty$ if the following conditions hold:
\begin{itemize}
\item for every 
$(u^M)_{M\in\R_+}\subset \U_\rho$ such that
$u^M \weak u$ as $M \to \infty$ we have
\begin{equation}\label{eq:Gamma_liminf}
\liminf_{M\to +\infty} \F^M_\rho(u^M) 
\geq \F_\rho(u);
\end{equation}
\item for every $u\in \U$ there exists a sequence
$(u^M)_{M\in\R_+}\subset \U_\rho$ such that
$u^M \weak u$ as $M\to \infty$ and such that
\begin{equation}\label{eq:Gamma_limsup}
\limsup_{M\to +\infty} \F^M_\rho(u^M) \leq \F_\rho(u).
\end{equation}
\end{itemize}
If \eqref{eq:Gamma_liminf} and \eqref{eq:Gamma_limsup}
are satisfied, then we write 
$\F^M_\rho\to_\Gamma \F_\rho$ as $M \to \infty$.
\end{defn*}
Let us define the functional $\F:\U\to\R$
as follows:
\begin{equation}\label{eq:G_limit}
\F(u) := \int_{\R^n} a(\Phi_u(x)-\Psi(x))\, d\mu
+\frac\beta2 ||u||_{L^2}^2,
\end{equation} 
where the probability measure $\mu$ is the weak limit
of the sequence $(\mu_M)_{M\geq1}$.
Using the same argument as in the proof of 
Proposition~\ref{prop:bound_norm}, we can prove that
$\F$ attains minimum and that
\[
\arg \min_{\U}\F = \arg \min_{\U_\rho}\F,
\]
with $\rho=C_\beta$. As before, we define
$\F_\rho:=\F|_{\U_\rho}$.

\begin{theorem}\label{thm:G_conv}
Given $\rho>0$, let us consider $\F^M_\rho:\U_\rho
\to \R$ with $M\geq1$. Let $\F_\rho:\U_\rho\to\R$
be the restriction to $\U_\rho$ of the functional 
defined in \eqref{eq:G_limit}. 
If Assumption~\ref{ass:weak_conv} holds, then the 
functionals $(\F^M_\rho)_{M\geq1}$ $\Gamma$-converge 
to $\F_\rho$ as $M\to\infty$ with respect to the weak 
topology of $\U$.
\end{theorem}

\begin{proof}
Let us prove the $\limsup$-condition 
\eqref{eq:Gamma_limsup}. Let us fix $u\in\U_\rho$
and, for every $M\geq1$, let us define 
$u^M:= u$. Then, recalling that the measures
$(\mu_M)_{M\geq1}$ are weakly convergent to $\mu$,
we have that
\[
\F_\rho(u)
= \lim_{M\to\infty}
\int_{\R^n} a(\Phi_u(x)-\Psi(x))\, d\mu_M
+\frac\beta2 ||u||_{L^2}^2
= \lim_{M\to\infty} \F^M_\rho(u).
\]
This proves \eqref{eq:Gamma_limsup}.

We now prove the $\liminf$-condition. Let 
$(u^M)_{M\geq1}\subset \U$ be weakly convergent
to $u\in\U$. 
Using Proposition~\ref{prop:conv_flows} we have that
$\Phi_{u^M}|_K\to \Phi_u|_K$ with respect to the
$C^0$ topology. 
Moreover, there exists a compact set $K'\subset \R^n$
such that
\[ (\Phi_{u}-\Psi)(K) \cup
\bigcup_{M\geq1}(\Phi_{u^M}-\Psi)(K)\subset K'.
\]
Using the fact that $a$ is uniformly continuous on
$K'$, we  deduce that
\begin{equation} \label{eq:conv_integrand}
\lim_{M\to\infty}
\sup_{x\in K}|a(\Phi_{u^M}(x)-\Psi(x) )-
a(\Phi_{u}(x)-\Psi(x) )|_2= 0.
\end{equation} 
Using \eqref{eq:conv_integrand} and the
weak convergence
of $(\mu_M)_{M\geq1}$ to $\mu$, we deduce that
\[
\lim_{M\to\infty} \int_{\R^n}
a(\Phi_{u^M}(x)-\Psi(x) ) \, d\mu_M
= \int_{\R^n}
a(\Phi_{u}(x)-\Psi(x) ) \, d\mu. 
\]
Recalling that the $L^2$ norm is lower semi-continuous 
with respect to the weak convergence,
we have that condition \eqref{eq:Gamma_liminf}
is satisfied. 
\end{proof}

\begin{remark} \label{rem:fund_Gamma}
Using the equi-coercivity of the functionals
$(\F^M_\rho)_{M\geq1}$ and 
\cite[Corollary~7.20]{D93}, we deduce that
\begin{equation} \label{eq:conv_minim}
\lim_{M\to \infty} \min_{\U_\rho} \F^M_\rho =
 \min_{\U_\rho} \F_\rho,
\end{equation}
and that any cluster point $\tilde u$ of 
a sequence of  minimizers
$(\tilde u^M)_{M\geq1}$ is a minimizer of $\F_\rho$.
Let us assume that a sub-sequence
$\tilde u^{M_j} \weak \tilde u$ as $j\to\infty$.
Using Proposition~\ref{prop:conv_flows} and
the Dominated Convergence Theorem we deduce that
\begin{equation}\label{eq:train_to_test}
\lim_{j\to\infty}
\int_K a(\Phi_{\tilde u^{M_j}}(x)-\Psi(x))\,d\mu_M(x)
=\int_K a(\Phi_{\tilde u}(x)-\Psi(x))\,d\mu(x),
\end{equation}
where we stress that $\tilde u$
 is a minimizer of $\F_\rho$.
Combining
\eqref{eq:conv_minim} and \eqref{eq:train_to_test}
we obtain that
\[
\lim_{j\to\infty}\frac\beta2 ||\tilde u^{M_j}||_{L^2}^2=
\frac\beta2 ||\tilde u||_{L^2}^2.
\]
Since $\tilde u^{M_j}\weak \tilde u$ as $j\to\infty$,
the previous 
equation implies that 
the subsequence $(\tilde u^{M_j})_{j\geq1}$
converges to $\tilde u$ also
with respect to the strong topology of $L^2$.
This argument shows that any sequence of
minimizers  $(\tilde u^M)_{M\geq1}$ is pre-compact
with respect to the strong topology of $L^2$.
\end{remark}

We can establish an asymptotic upper bound for the
mean training error. Let us define
\begin{equation}\label{eq:upp_train}
\kappa_\beta := \sup \left\{
\int_K a(\Phi_{\tilde u}(x)-\Psi(x))\,d\mu(x)
\bigg| \tilde u \in \arg  \min_{\U} \F
\right\}.
\end{equation}
As suggested by the notation, the value of 
$\kappa_\beta$ highly depends on the 
positive parameter $\beta$ that tunes the
$L^2$-regularization.
Given a sequence of minimizers 
$(\tilde u^M)_{M\geq 1}$ of the functionals 
$(\F^M)_{M\geq 1}$, from \eqref{eq:train_to_test} 
we deduce that
\begin{equation} \label{eq:asymp_up_train}
\limsup_{M\to\infty} \int_K a(\Phi_{\tilde u^M}(x)-
\Psi(x))\,d\mu_M(x) \leq \kappa_\beta.
\end{equation}
In the next result we show that
under the hypotheses of 
Theorem~\ref{thm:univ_app_flows}
 $\kappa_\beta$
can be made arbitrarily small with a proper choice
of $\beta$.

\begin{proposition} \label{prop:param_beta}
Let $\kappa_\beta$ be defined as in 
\eqref{eq:upp_train}. If the vector fields
$F_1,\ldots,F_l$ that define control system
\eqref{eq:subr_sys_notat} satisfy 
Assumption~\ref{ass:bound_fields}
and~\ref{ass:str_Lie}, then
\begin{equation} \label{eq:beta_to_0}
\lim_{\beta \to 0^+} \kappa_\beta =0.
\end{equation}
\end{proposition}
 
\begin{proof}
Let us fix $\e>0$. Since $a(0)=0$, there exists
$\rho>0$ such that 
\[
\sup_{B_\rho (0)} a \leq \e.
\]
Using Theorem~\ref{thm:univ_app_flows}, we deduce
that there exists a control $\hat u \in \U$
such that
\begin{equation} \label{eq:unif_app_proof}
\sup_{x\in K}|\Phi_{\hat u}(x)-\Psi(x)|_2 < \rho.
\end{equation}
This implies that
\[
\int_K a(\Phi_{\hat u}(x)-\Psi(x))\,d\mu(x) \leq \e.
\]
Let us set 
$\hat \beta:=\frac{2\e}{||\hat u||_{L^2}^2}$.
For any $\beta\leq \hat \beta$, let 
$\F$ be the functional defined in \eqref{eq:G_limit}
with tuning parameter $\beta$, and let $\tilde u
\in \U$ be a minimizer of $\F$. Then we have
\[
\F(\tilde u) \leq \F(\hat u) \leq 
\e + \frac{\beta}{2}||\hat u||_{L^2}^2\leq 2\e,
\] 
and this concludes the proof.
\end{proof} 
 
\subsection{An estimate of the generalization error}
We now discuss an estimate
of the expected generalization error
based on the observation of the mean training error,
similar to the one established in
\cite{MTG21} for the control system 
\eqref{eq:ctrl_sys_ResNet}.
A similar estimate was obtained also in 
\cite{BCFH21}.
Assumption~\ref{ass:weak_conv} implies that
the Wasserstein distance $W_1(\mu_M,\mu)\to0$
as $M\to\infty$ (for details, see
\cite[Proposition~7.1.5]{AGS}).
We recall that, if $\nu_1,\nu_2 \in 
\mathcal{P}(K)$ are Borel
probability measures on $K$, then
\[
W_1(\nu_1,\nu_2):= \inf_{\pi\in \mathcal{P}(K\times K)}
\left\{ \int_{K\times K}\!\!\!\!\!\!|x-y|_2\,d\pi(x,y) 
\Big| \pi(\cdot,K)=\nu_1,
\pi(K,\cdot) = \nu_2 
\right\}.
\]
For every $M\geq1$ let us introduce
$\pi_M\in \mathcal{P}(K\times K)$
such that $\pi_M(\cdot,K)=\mu_M$
and  $\pi_M(K,\cdot)=\mu$, and
\[
W_1(\mu_M,\mu)=
\int_{K\times K}\!\!\!\!\!\!|x-y|_2\,d\pi_M(x,y) .
\]
{Let us consider an admissible 
control $u\in\U$, and let $\Phi_u:\R^n\to\R^n$ be 
the corresponding flow. }
If the testing samples are generated using the 
probability distribution $\mu$, then the expected
generalization error {that we commit
by approximating $\Psi:\R^n\to\R^n$ with 
$\Phi_u$ is
\begin{equation*}
\mathbb{E}_\mu [a(\Phi_{ u}(\cdot)-
\Psi(\cdot))]
= \int_K a(\Phi_{u}(y)-\Psi(y)) \,d\mu(y).
\end{equation*} 
On the other hand,  we recall that
the corresponding training error
is expressed by
\begin{equation*}
\int_K a(\Phi_u(x)-\Psi(x))\, d\mu_M(x).
\end{equation*} 
Hence we can compute
\begin{align*}
\bigg| \mathbb{E}_\mu [a(\Phi_{u}(\cdot)-
\Psi(\cdot))] &- \int_Ka(\Phi_{u}(x)-\Psi(x)) \,d
\mu_M(x) \bigg|  \\
&\leq   \int_{K\times K}\!\!\!\!\!\!
|a(\Phi_{u}(y)-\Psi(y))-
a(\Phi_{u}(x)-\Psi(x)) |\,d\pi_M (x,y) \\
& \leq L_a \int_{K\times K} \!\!\!\!\!\!
 |\Psi(y)-\Psi(x)| +
|\Phi_{u}(x)-\Phi_{u}(y)|
\,d\pi_M(x,y).
\end{align*} 
Then for every $M\geq1$ we have
\begin{equation*}
\bigg| \mathbb{E}_\mu [a(\Phi_{u}(\cdot)-
\Psi(\cdot))]
 - \int_Ka(\Phi_{u}(x)-\Psi(x)) \,d
\mu_M(x) \bigg| 
\leq L_a(L_\Psi + L_{\Phi_{u}})W_1(\mu_M,\mu),
\end{equation*}
where $L_\Psi$, $L_{\Phi_{u}}$ and $L_a$
are the Lipschitz
constant, respectively, of $\Psi$, $\Phi_{u}$ 
and $a$}.
{The last inequality finally yields
\begin{equation}\label{eq:est_gen_err_univers}
\mathbb{E}_\mu [a(\Phi_{u}(\cdot)-
\Psi(\cdot))] \leq
\int_Ka(\Phi_{u}(x)-\Psi(x)) \,d
\mu_M(x) +
L_a(L_\Psi + L_{\Phi_{u}})W_1(\mu_M,\mu)
\end{equation}
for every $M\geq 1$.} 

\begin{remark} \label{rmk:a_priori_est}
We observe that the estimate 
\eqref{eq:est_gen_err_univers} does not 
involve any testing dataset. In other words,
in principle we can use 
\eqref{eq:est_gen_err_univers} to provide an upper
bound to the expected generalization
error, without the need of computing the 
mismatch between $\Psi$ and $\Phi_u$ on a 
testing dataset.
In practice, while we can directly measure 
the first quantity at the right-hand side of
\eqref{eq:est_gen_err_univers}, the second term
could be challenging to estimate.
Indeed, if on one hand we can easily approximate
the quantity $L_{\Phi_u}$ (for instance by means of
\eqref{eq:lipsc_flow}), on the other hand
we may have no access to the distance 
$W_1(\mu_M,\mu)$. This is actually the case when the
measure $\mu$ used to sample the training dataset
is unknown.
\end{remark}

{In the case we consider the flow 
$\Phi_{\hat u^M}$ corresponding to a minimizer
$\hat u^M$ of the functional $\F^M$ we can 
further specify \eqref{eq:est_gen_err_univers}.}
Indeed, combining
 Proposition~\ref{prop:bound_norm} and
Lemma~\ref{lem:lipsch_flows}, we deduce that
$L_{\Phi_{\tilde u^M}}$ is uniformly bounded
with respect to $M$ by a constant $L_\beta$.
Provided that  $M$ is large enough,
 from \eqref{eq:est_gen_err_univers} 
and  \eqref{eq:asymp_up_train}
we obtain that
\begin{equation} \label{eq:est_gen_err}
\mathbb{E}_\mu [a(\Phi_{\tilde u^M}(\cdot)-
\Psi(\cdot))] \leq 2\kappa_\beta + 
L_a(L_\Psi + L_\beta)W_1(\mu_M,\mu).
\end{equation}
The previous inequality shows how we can achieve
an arbitrarily small expected generalization error,
provided that the vector fields $F_1,\ldots,F_l$
of control system \eqref{eq:subr_sys_notat}
satisfy Assumption~\ref{ass:str_Lie},
and provided that the size of the training
dataset could be chosen arbitrarily large.
First, using Proposition~\ref{prop:param_beta}
 we set the tuning parameter $\beta$ such 
that the quantity $\kappa_\beta$ is as small
as desired. Then, we consider a training dataset
large enough to guarantee that the second term
at the right-hand side of \eqref{eq:est_gen_err}
is of the same order of $\kappa_\beta$.

\begin{remark} \label{rmk:beta_eps_choice}
Given $\e>0$,
Proposition~\ref{prop:param_beta}  
guarantees the existence of 
$\hat \beta>0$ such that $\kappa_\beta\leq\e$
if $\beta\leq \hat\beta$. The expression
of $\hat\beta$ obtained in the proof of
Proposition~\ref{prop:param_beta} is given in  
terms of the norm of a control $\hat u\in \U$
such that \eqref{eq:unif_app_proof} holds.
In \cite{AS1},
where Theorem~\ref{thm:univ_app_flows}
is proved, it is explained the construction of
an admissible control  whose flow approximate
the target diffeomorphism with assigned precision.
However the control produced with this procedure
is, in general, far from having minimal $L^2$-norm,
and as a matter of fact the corresponding 
$\hat{\beta}$ might be smaller than necessary.
Unfortunately, at the moment, we can not provide a 
more practical rule for the computation of 
$\hat \beta$.
\end{remark}

\begin{remark} \label{rmk:a_priori_beta}
As observed in Remark~\ref{rmk:a_priori_est}
 for \eqref{eq:est_gen_err_univers},
the estimate
\eqref{eq:est_gen_err} of the expected generalization
error holds as well
 \textit{a priori} with respect to
the choice of a testing dataset.
Moreover, if the size $M$ of the training dataset
is assigned and it cannot be enlarged,
in principle one could choose the
regularization parameter $\beta$ by minimizing
the right-hand side of \eqref{eq:est_gen_err}.
However, in practice this may be very complicated,
since we have no direct access to the function
$\beta\mapsto \kappa_\beta$.
\end{remark}
 
\end{section}

\begin{section}{Training the Network:
Projected Gradient Flow} \label{sec:proj_grad}
In this section we propose a training algorithm
based on a gradient flow formulation.
We recall that we want to express the approximation
of a diffeomorphism $\Psi:\R^n\to\R^n$
diffeotopic to the identity through a
proper composition $\Phi=\Phi_N \circ\ldots\circ
\Phi_1$, where, for every $k=1,\ldots,N$ the
functions $\Phi_k$ are of the form 
\begin{equation} \label{eq:fun_layer}
\Phi_k(x) = x + h\sum_{i=1}^lF_i(x)u_{i,k},
\quad h=\frac1N,
\end{equation}
and play the role of inner layers of the ResNet.
We recall that the building blocks
\eqref{eq:fun_layer} are obtained by discretizing
the linear-control system \eqref{eq:subr_sys_notat}
with the explicit Euler scheme on the evolution
interval $[0,1]$, using a constant time-step
$h=\frac1N$.
In Section~\ref{sec:def_strat} and 
Section~\ref{sec:G-conv} we observed that
the minimization of the non-linear functional 
$\F^M:\U\to\R$ defined
in \eqref{eq:fun_fin} provides a possible
strategy for approximating $\Psi$ using a flow
$\Phi_u$ generated by control system 
\eqref{eq:subr_sys_notat}. The idea that underlies 
this section consists in projecting the gradient
flow induced in $\U$ by $\F^M$ onto a finite
dimensional subspace $\U_N\subset\U$. Using this
gradient flow, we obtain a procedure to learn
the parameters 
$(u_{i,k})_{k=1\ldots,N}^{i=1,\ldots,l}$
that define the inner layers.

In \cite{S21} it has been derived a gradient flow
equation for optimal control problems with
end-point cost, and a convergence result
was established.
In the particular case of the functional $\F^M$
defined in \eqref{eq:fun_fin} and associated
to control system \eqref{eq:subr_sys_notat},
from \cite[Remark~3.5]{S21}
 it follows that the 
differential 
$d_u\F^M=(\partial_{u_1}\F^M
,\ldots,\partial_{u_l}\F^M)$
can be represented as an element of $\U$ as follows:
\begin{equation}\label{eq:diff_fun}
\frac{\partial}{\partial u_i}\F^M(u)(s) = 
\sum_{j=1}^M\langle
\lambda_u^{j}(s),F_i(x_u^{j}(s)) \rangle
+ \beta u_i(s), \quad i=1,\ldots,l.
\end{equation}
For every $j=1,\ldots,M$ and a.e.
$s\in[0,1]$, the functions
$x_u^{j}:[0,1]\to\R^n$ and 
$\lambda_u^{j}:[0,1]\to\R^n$
satisfy
\begin{equation} \label{eq:ode_traj}
\begin{cases}
\dot x_u^{j}(s)= \sum_{i=1}^lF_i(x_u^{j}(s))u_i(s), \\
x^{j}_u(0)= x^{j}_0,
\end{cases}
\end{equation}
and
\begin{equation}
\begin{cases} \label{eq:ode_covect}
\dot \lambda_u^{j}(s)= 
-\lambda_u^{j}(s)
\left(\sum_{i=1}^l u_i(s)
 D_xF_i(x^{j}(s))
\right), \\
\lambda_u^{j}(1)
= \frac{1}{M}\nabla a(x_u^{j}(1)-\Psi(x^{j}_0)).
\end{cases}
\end{equation}
A remarkable aspect that can be used in the numerical
implementation is that the solutions of 
\eqref{eq:ode_traj} and \eqref{eq:ode_covect} can be
computed in parallel with respect to the elements
of the training ensemble,
i.e., in other words, with respect to
the index $j$.
In this section we propose a Finite-Elements approach
to derive a method for training the ResNet 
\eqref{eq:fun_layer}.
The idea is to replace the infinite-dimensional space
$\U$ with a proper finite-dimensional subspace
$\U_N$. A natural choice could be to consider 
piecewise constant functions, i.e.,
\begin{equation}
u\in\U_N \iff \exists c_1,\ldots,c_N \in \R^l: 
u(s) =
\begin{cases}
c_1 & s\in[0,\frac1N],\\
\ldots\\
c_N & s\in (\frac{N-1}{N},1].
\end{cases}
\end{equation}
The orthogonal projection onto $\U_N$,
$P_N:\U\to\U_N$ is given by
\begin{equation} \label{eq:proj_expr}
P_N(v)(s) =
\begin{cases}
N\int_{[0,\frac1N]}v(\tau)\,d\tau & s\in[0,\frac1N],\\
\ldots\\
N\int_{(\frac{N-1}{N},1]}v(\tau)\,d\tau & s\in (\frac{N-1}{N},1].
\end{cases}  
\end{equation}
Hence, the gradient flow of the functional $\F^M$
\[
\partial_t u = -d\F^M(u)
\]
can be projected onto the finite-dimensional space
$\U_N$:
\begin{equation} \label{eq:proj_grad_flow}
\partial_t u = -P_N(d\F^M(u)).
\end{equation}
With a discretization-in-time of 
\eqref{eq:proj_grad_flow}, we obtain the following
iterative implementable method for the 
approximate minimization of $\F^M$:
\begin{enumerate}
\item[(i)] Consider $u\in \U_N$, and, for every
$j=1,\ldots,M$, compute an approximation of 
$x_u^{j}, \lambda_u^{j}$ at the nodes
$\{ 0,\frac1N,\ldots,1 \}\subset[0,1]$ using a proper
numerical scheme for ODE;
\item[(ii)] Using the approximations of 
$\{x_u^{j}, \lambda_u^{j}\}_{j=1,\ldots,M}$
obtained at step (i), compute the approximation
of $d\F^M(u)$ at the nodes using \eqref{eq:diff_fun};
\item[(iii)] Compute $\Delta u$, i.e.,
the approximation of 
$P_N(d\F^M(u))$ using \eqref{eq:proj_expr}, 
the approximation of $d\F^M(u)$ obtained at step (ii)
and a proper numerical integration scheme;
\item[(iv)] Update the control $u$ using an explicit
Euler discretization of \eqref{eq:proj_grad_flow}
with a proper step-size $\gamma>0$:
\[
u = u - \gamma\Delta u.
\]
Finally go to step (i).
\end{enumerate}
The step-size $\gamma>0$ can be chosen using
backtracking line search.
In Algorithm~\ref{alg:proj_grad_flow}
we present the implementation that we used for
the numerical simulations.
We use the following notations:
\begin{itemize}
\item $u=(u_{i,k})^{i=1,\ldots,l}_{ k=1,\ldots,N}$ 
denotes an element of $\U_N$. Namely,
$u_{i,k}$ represents the value of the control 
corresponding to the vector field $F_i$ in the
time interval 
$\left[\frac{k-1}{N},\frac{k}{N}\right)$.
The same notations are used for 
$\Delta u=(\Delta u_{i,k})^{i=1,\ldots,l}_{ k=1,
\ldots,N}$, which represents the approximation
of $P_N(d\F^M(u))$. 

\item For every ${j=1,\ldots,M}$, 
$x^j=(x^j_k)_{ k=0,\ldots,N}$
denotes the approximated solution of 
\eqref{eq:ode_traj} with Cauchy datum 
$x^j(0) =x^j_0$. Namely, $x^j_k$ represents the
approximation at the node 
$t=\frac{k}{N}$. 
We use $x=(x^j)_{j=1,\ldots,M}$ to 
denote the array containing the approximations
of all trajectories at all nodes.
Moreover, we observe that 
we have $x^j_k = \Phi_k\circ\ldots\circ\Phi_1
(x^j_0)$ for every $k=1,\ldots,N$ and
for every $j=1,\ldots,M$, where 
$\Phi_1,\ldots,\Phi_N$ are the maps introduced in
\eqref{eq:fun_layer}.

\item For every ${j=1,\ldots,M}$, 
$\lambda^j=(\lambda^j_k)_{ k=0,\ldots,N}$
denotes the approximated solution of 
\eqref{eq:ode_covect} with Cauchy datum 
$\lambda^j(1) = \frac1M \nabla a (x^j_N - \Psi(x^j_0))$.
Namely, $\lambda^j_k$ represents the
approximation at the node $t=\frac{k}{N}$.
We use $\lambda=(\lambda^j)_{j=1,\ldots,M}$ to 
denote the array containing the approximations
of all co-vectors at all nodes.

\end{itemize}

\begin{algorithm}
\KwData{ $(x_0^j)_{j=1,\ldots,M}\subset \R^n$ training 
dataset, $(F_i)_{i=1,\ldots,l}$  vector fields.\\ 
Set the parameters: $n_{\mathrm{layers}}\geq 1$, $\tau\in (0,1)$, $c \in (0,1)$, $
\gamma >0$, $\max_{\mathrm{iter}}\geq 1$.}
$N\gets n_{\mathrm{layers}}$\;
$h\gets\frac{1}{N}$\;
$u \in \U_N$\;
\For(\tcp*[f]{Forward solution of \eqref{eq:ode_traj}}){$j=1,\ldots,M$ }{
	\For{$k=1,\ldots,N$}{
		$x^j_{k}\gets x^j_{k-1} + h
			\sum_{i=1}^l F_i(x_{k-1}^j) u_{i,k}$}
	}

$\mathrm{Cost}\gets \frac1M\sum_{j=1}^M
a(x^j_N - \Psi(x^j_0)) + \frac\beta2 ||u||_{L^2}^2$\;
$\mathrm{flag}\gets 1$\;
\For(\tcp*[f]{Iterations of Projected Gradient Flow}){$r=1,\ldots,\max_{\mathrm{iter}}$ }{
	\If(\tcp*[f]{Solve \eqref{eq:ode_covect} only if necessary}){$\mathrm{flag}=1$ }{
		\For(\tcp*[f]{Backward solution
of \eqref{eq:ode_covect} }){$j=1,\ldots,M$ }{
			$\lambda^j_N\gets \frac{1}{M}\nabla a(x^j_N 
- \Psi(x^j_0))$\;
			\For{$k=N,\ldots,1$}{
				$\lambda^j_{k-1} \gets 
(\mathrm{Id}- h \sum_{i=1}^l D_xF_i(x_{k-1}^j) 
u_{i,k})^{-T}\lambda_k^j$\;
				}			
			}
		}
	\For(\tcp*[f]{Compute $P_N(d\F^M)$ using \eqref{eq:proj_expr}}){$k=1,\ldots,N$, $i=1,\ldots,l$ 
	}{
		$\Delta u_{i,k} \gets  \sum_{j=1}^M
(\frac12\langle \lambda_{k-1}^j, F_i(x_{k-1}^j) \rangle
+ \frac12\langle \lambda_{k}^j, F_i(x_{k}^j) 
\rangle) + \beta u_{i,k}$\;
		$u_{i,k}^{\mathrm{new}}\gets u_{i,k} - \gamma \Delta u_{i,k}$\;
		}
		\For(\tcp*[f]{Forward solution 
		of \eqref{eq:ode_traj}}){$j=1,\ldots,M$ }{
			$x^{j,\mathrm{new}}_0\gets x^j_0$\;			
			\For{$k=1,\ldots,N$}{
				$x^{j,\mathrm{new}}_{k}\gets 
x^{j,\mathrm{new}}_{k-1} + h
\sum_{i=1}^l F_i(x_{k-1}^{j,\mathrm{new}}) u_{i,k}^{\mathrm{new}}$\;
				}			
			}
	$\mathrm{Cost^{new}}
\gets \frac1M\sum_{j=1}^M
a(x^{j,\mathrm{new}}_N - \Psi(x^j_0)) + \frac\beta2 ||u^{\mathrm{new}}||_{L^2}^2$\;	
	\eIf(\tcp*[f]{Backtracking for $\gamma$}){$\mathrm{Cost}\geq \mathrm{Cost^{new}}
+ c\gamma || \Delta u ||_{L^2}^2$ }{
		$u\gets u^{\mathrm{new}}$,
		$x\gets x^{\mathrm{new}}$\;
		$\mathrm{Cost}\gets \mathrm{Cost^{new}}$\;
		$\mathrm{flag} \gets 1$\;
		}
		{
		$\gamma \gets \tau \gamma$\;
		$\mathrm{flag} \gets 0$\;
		}
	}
\caption{Training with Projected Gradient Flow}
\label{alg:proj_grad_flow}
\end{algorithm}

\begin{remark}
We briefly comment 
Algorithm~\ref{alg:proj_grad_flow}.
Both \eqref{eq:ode_traj} and \eqref{eq:ode_covect} are
discretized using the explicit Euler method. However,
since \eqref{eq:ode_covect} is solved backward in 
time, the update rule
\begin{equation*} 
\frac{\lambda^j_{k} - \lambda^j_{k-1}}{h}
= -\lambda_{k-1}^j 
\left(
\sum_{i=1}^lD_xF_i(x^j_{k-1}) u_{i,k} 
\right)
\end{equation*}
is implicit in the variable $\lambda_{k-1}^j$,
 and it
requires the solution of a $n\times n$ linear system,
namely
\begin{equation} \label{eq:update_rule_bacward}
\lambda^j_{k-1} 
=
\lambda_{k}^j
\left( \mathrm{Id} -  h\sum_{i=1}^lD_xF_i(x^j_{k-1}) u_{i,k}\right)^{-1},
\end{equation}
as done in line 16.
We recall that 
$(\lambda^j_k)^{j=1,\ldots,M}_{k=0,\ldots,N}$ should 
be understood as row vectors.

\noindent  
In line 21 we used the trapezoidal rule to 
approximate
\begin{equation*}
N\int_{\frac{k-1}{N}}^{\frac{k}{N}}
 \sum_{j=1}^M\langle
\lambda_u^j(s),F_i(x_u^j(s))
\rangle \, ds \simeq
\frac{1}{2} \sum_{j=1}^M
\left(
\langle \lambda_{k-1}^j,F_i(x^j_{k-1}) \rangle+
\langle \lambda_k^j,F_i(x^j_k) \rangle
\right).
\end{equation*}
This choice seems natural since the function inside
the integral is available at the nodes
$\frac{k-1}{N}$ and $\frac{k}{N}$, for every 
$k=1,\ldots,N$. 
\end{remark}

\begin{remark}
It is interesting to compare the update of the
controls prescribed by
 Algorithm~\ref{alg:proj_grad_flow}
with the standard backpropagation of the gradients.
Given the parameters
$(u_{i,k})^{i=1,\ldots,l}_{ k=1,\ldots,N}$,
for every $k=1,\ldots,N$ let $\Phi_k:\R^n\to\R^n$
be the $k$-th building block obtained by using 
$(u_{i,k})^{i=1,\ldots,l}$ in \eqref{eq:fun_layer},
and let $\Phi=\Phi_N\circ\ldots\circ\Phi_1$
be the corresponding composition. 
For every $j=1,\ldots, M$, the quantity 
$C^j =\frac1M a(\Phi(x_0^j) - \Psi(x_0^j))
=\frac1M a(x_N^j-\Psi(x_0^j))$ represents the
the training error relative to the point $x_0^j$
in the dataset.
A direct computation shows that
\begin{align*}
\frac{\partial C^j}{\partial u_{\bar i, \bar k}}=&
\frac1M \nabla a(x_N^j - \Psi(x_0^j))
\cdot 
\left( \mathrm{Id} +  h\sum_{i=1}^lD_xF_i(x^j_{N-1}) u_{i,N}\right) \cdot \ldots \\
& \quad \ldots \cdot  
\left( \mathrm{Id} +  h\sum_{i=1}^lD_xF_i(x^j_{\bar k}) u_{i,\bar k+1}\right)\cdot F_{\bar i}(x_{\bar k -1}^j),
\end{align*}
where we used the notation $x^j_k = \Phi_k\circ
\ldots\circ\Phi_1(x_0^j)$. The last identity
can be rewritten as 
\begin{equation*}
\frac{\partial C^j}{\partial u_{\bar i, \bar k}}=
\langle \tilde \lambda^j_{\bar k }, F_{\bar i}(x_{\bar k-1}^j)
\rangle,
\end{equation*}
where $\tilde \lambda_{N}^j,\ldots, \tilde \lambda_{\bar k }^j$ are recursively defined
as follows
\begin{equation} \label{eq:lambda_backprop}
\begin{cases}
\tilde \lambda_{N}^j = \frac1M \nabla a (x_N^j -\Psi(x_0^j)),
\\
\tilde \lambda_{k-1}^j = \tilde \lambda_{k}^j 
\left( \mathrm{Id} +  h\sum_{i=1}^lD_x
F_i(x^j_{ k-1}) u_{i, k}\right) &
k=N,\ldots,\bar k +1.
\end{cases}
\end{equation}
Therefore, from \eqref{eq:lambda_backprop} we 
deduce that the backpropagation is related to
the discretization of \eqref{eq:ode_covect}
via the implicit Euler scheme:
\begin{equation*} 
\frac{\tilde\lambda^j_{k} -
\tilde \lambda^j_{k-1}}{h}
= -\tilde\lambda_{k}^j 
\left(
\sum_{i=1}^lD_xF_i(x^j_{k-1}) u_{i,k} 
\right).
\end{equation*}
However, since the Cauchy problem 
\eqref{eq:ode_covect} is solved backward in time, 
the previous expression is explicit in 
$\tilde \lambda^j_{k-1}$.
\end{remark}

\begin{remark}
In Algorithm~\ref{alg:proj_grad_flow}
all the {\it for} loops that involve the 
index variable of the elements of the ensemble
(i.e., the counter $j=1,\ldots,M$)
 can be carried out
in parallel. This is an important fact when dealing
with large dataset, since these loops are involved
in the time-consuming procedures that solve
\eqref{eq:ode_traj} and \eqref{eq:ode_covect}. 
However, in order to
update the controls (lines 20-23), 
we need to wait for the conclusion of
these computations and to
recollect the results.
\end{remark}

\begin{remark}
In Algorithm~\ref{alg:proj_grad_flow} at each 
iteration we use the entire dataset to update
the controls. However, it is possible to 
consider mini-batches, i.e., to sample at 
the beginning of each iteration a 
subset of the training points $\{ x_0^1,\ldots,
x_0^M \}$ of size $M'<M$. This practice is very
common in Deep Learning, where it is known as
\textit{stochastic gradient descent} (see, e.g., 
\cite[Section~5.9]{GBC16}).
We stress that, in case of stochastic mini-batches,
the cost used to accept/reject the update
should be computed with respect to the reduced 
dataset, and not to the full one. More precisely,
we should compare the quantity
\begin{equation*}
\mathrm{Cost}' = \frac1{M'} \sum_{j'=1}^{M'}
a(x^{j'}_N - \Psi(x_0^{j'})) + \frac\beta2
||u||_{L^2}^2
\end{equation*}
before and after the update of the controls.
\end{remark}

\end{section}

\begin{section}{Training the Network: Maximum 
Principle} \label{sec:train_PMP}
In this section we study the same problem as in
Section~\ref{sec:proj_grad}, i.e., we want to
define an implementable algorithm to learn
the parameters that define the functions
\eqref{eq:fun_layer}, which
play the role of inner layers in our ResNet.

As in the previous section,
we approach the training problem by
minimizing the functional $\F^M$.
However, in this case, the minimization 
procedure is based on the Pontryagin Maximum 
Principle. 
We state below the Maximum Principle
for our particular Optimal Control problem.
For a detailed and general presentation
of the topic the reader is referred to the
textbook \cite{AS}.

\begin{theorem}\label{thm:PMP}
Let $\tilde u\in \U$ be an admissible control that
minimizes the functional $\F^M$ defined in 
\eqref{eq:fun_M}. 
Let 
$\mathcal{H}:(\R^n)^M\times(\R^n)^M\times\R^l\to\R$
be the hamiltonian function defined as follows:
\begin{equation}\label{eq:ham_PMP}
\mathcal{H}(x,\lambda,u) =\sum_{j=1}^M
 \langle
\lambda^{j}, F(x^{j})u 
\rangle
- \frac\beta2|u|^2,
\end{equation}
where we set $x=(x^{1},\ldots,x^{M})$
and $\lambda=(\lambda^{1},\ldots,\lambda^{M})$.
Then there exists an absolutely continuous
function $\lambda_{\tilde u}:[0,1]\to(\R^n)^M$ such
that the following conditions hold:
\begin{itemize}
\item For every $j=1,\ldots,M$  the curve
$x_{\tilde u}^{j}:[0,1]\to\R^n$ satisfies
\begin{equation}\label{eq:ode_traj_PMP}
\begin{cases}
\dot x_{\tilde u}^{j}(s) =  \frac{\partial
}{\partial \lambda^{j}}\mathcal{H}(x_{\tilde u}(s),
\lambda_{\tilde u}(s),\tilde u(s)),\\
x_{\tilde u}^{j}(0) = x^{j}_0;
\end{cases}
\end{equation}
\item For every $j=1,\ldots,M$  the curve
$\lambda_{\tilde u}^{j}:[0,1]\to\R^n$ satisfies
\begin{equation}\label{eq:ode_covec_PMP}
\begin{cases}
\dot \lambda_{\tilde u}(s) = - \frac{\partial
}{\partial x^{j}}\mathcal{H}(x_{\tilde u}(s),
\lambda_{\tilde u}(s),\tilde u(s)),\\
\lambda_{\tilde u}^{j}(0) = - \frac1M
\nabla a(x_{\tilde u}^{j}(1)-\Psi(x^{j}_0));
\end{cases}
\end{equation}
\item For every $s\in[0,1]$, the following condition
is satisfied:
\begin{equation}\label{eq:max_cond_PMP}
\tilde u(s) \in \arg \max_{u\in \R^l}
\mathcal{H}(x_{\tilde u}(s), \lambda_{\tilde u}(s),u).
\end{equation}
\end{itemize}
\end{theorem}

\begin{remark}
In Theorem~\ref{thm:PMP} we stated the Pontryagin 
Maximum Principle for normal extremals only. This
is due to the fact that the optimal control
problem concerning the minimization of $\F^M$
does not admit abnormal extremals.
\end{remark}

An iterative method based on the Maximum Principle
for the numerical resolution of Optimal Control 
problems was proposed in \cite{CL}. 
The idea of training a deep neural network
with a numerical method
designed for optimal control problems
is already present in \cite{LCTE17} and
in \cite{BC19}, where the control system is of the
form \eqref{eq:ResNet_intro}.
In \cite{LCTE17} the authors
introduced a stabilization of the iterative method
described in \cite{CL}, while in \cite{BC19}  
the authors considered symplectic discretizations of
\eqref{eq:ode_traj_PMP}-\eqref{eq:ode_covec_PMP}.
Finally, in \cite{BCFH21} the authors used the
{\it particle method} to solve numerically the 
Mean Field Pontryagin equations, and employed a
root-finding procedure to maximize the
Hamiltonian.
The iterative method that we use in this section
was proposed in \cite{SS80}, and once again it is 
a stabilization of a method described in \cite{CL}.
The procedure consists in the following steps:
\begin{enumerate}
\item[(a)] Consider a control $u\in\U$
and compute $x_u:[0,1]\to(\R^n)^M$ using 
\eqref{eq:ode_traj_PMP};
\item[(b)] Assign $u_{\mathrm{old}}\gets u$ and
compute $\lambda_{u_\mathrm{old}}:[0,1]\to(\R^n)^M$ 
solving backward
\eqref{eq:ode_covec_PMP};
\item[(c)] Determine
$u_{\mathrm{new}}$ such that, for every $s\in[0,1]$,
\begin{equation} \label{eq:disc_max_PMP}
u_{\mathrm{new}}(s)\in \arg \max_{u\in\R^l}
\left( H(x_{u_{\mathrm{new}},}(s),\lambda_{u_\mathrm{old}}(s),
u) + \frac{1}{2\gamma}|u-u_{\mathrm{old}}(s)|^2 \right),
\end{equation}
where $x_{u_{\mathrm{new}}}:[0,1]\to(\R^n)^M$ solves
\eqref{eq:ode_traj_PMP} with respect to the control
$u_{\mathrm{new}}$.

\item[(d)] 
If $\F^M(u_\mathrm{new})<\F^M(u_\mathrm{old})$, then
update $u\gets u_\mathrm{new}$ and go to Step~(b).
If $\F^M(u_\mathrm{new})\geq \F^M(u_\mathrm{old})$,
then decrease the parameter $\gamma>0$, and go to
Step~(c).
\end{enumerate}
The algorithm as described above is not implementable,
since the maximization in Step~(c) requires the
availability of $x_{u_{\mathrm{new}}}$. 
However, in practice, 
the maximization of the extended Hamiltonian
in \eqref{eq:disc_max_PMP} and the computation of
the updated trajectory take place in subsequent steps.
Namely, after introducing the finite dimensional
subspace $\U_N\subset \U$ as in 
Section~\ref{sec:proj_grad}, 
for every $k=1,\ldots,N$ we plug into
\eqref{eq:disc_max_PMP} the approximations of
$\lambda_{u_{\mathrm{old}}}$ and 
$x_{u_{\mathrm{new}}}$ at the node $\frac{k-1}{N}$.
The value of the control $u$ in the interval 
$\left[ \frac{k-1}{N}, \frac{k}{N} \right)$
is set as the value where the maximum
of \eqref{eq:disc_max_PMP} is attained, and we use
it to compute $x_{u_{\mathrm{new}}}$ in the node
$\frac{k}{N}$. 
\begin{remark}
In \cite{SS80} the authors proved that 
the cost of the optimal control problem is decreasing
along the sequence of controls produced by the method,
provided that $\gamma$ is small enough.
Actually, this monotonicity result
is valid only for the ``ideal"
continuous-time procedure, and not for the 
discrete-time implementable algorithm.
A similar issue is observed for the iterative 
algorithm proposed in \cite{LCTE17}.
Anyway, the condition in Step~(d) is useful 
in the discrete-time case 
to adaptively adjust $\gamma$. 
\end{remark} 
  
We report in Algorithm~\ref{alg:max_princ}
the implementation of the procedure described above.
We use the same notations as in 
Section~\ref{sec:proj_grad}.

\begin{algorithm}
\KwData{ $(x_0^j)_{j=1,\ldots,M}\subset \R^n$ training 
dataset, $(F_i)_{i=1,\ldots,l}$  vector fields.\\ 
Set the parameters: $n_{\mathrm{layers}}\geq 1$, $\tau\in (0,1)$, $c \in (0,1)$, $
\gamma >0$, $\max_{\mathrm{iter}}\geq 1$.}
$N\gets n_{\mathrm{layers}}$\;
$h\gets\frac{1}{N}$\;
$u \in \U_N$\;
\For(\tcp*[f]{Forward solution of \eqref{eq:ode_traj_PMP}}){$j=1,\ldots,M$ }{
	\For{$k=1,\ldots,N$}{
		$x^j_{k}\gets x^j_{k-1} + h
			\sum_{i=1}^l F_i(x_{k-1}^j) u_{i,k}$}
	}

$\mathrm{Cost}\gets \frac1M \sum_{j=1}^M
a(x^j_N - \Psi(x^j_0)) + \frac\beta2 ||u||_{L^2}^2$\;
$\mathrm{flag}\gets 1$\;
\For(\tcp*[f]{Iterations of Projected Gradient Flow}){$r=1,\ldots,\max_{\mathrm{iter}}$ }{
	\If(\tcp*[f]{Solve \eqref{eq:ode_covec_PMP} only if necessary}){$\mathrm{flag}=1$ }{
		\For(\tcp*[f]{Backward solution
of \eqref{eq:ode_covec_PMP} }){$j=1,\ldots,M$ }{
			$\lambda^j_N\gets -\frac1M\nabla a(x^j_N 
- \Psi(x^j_0))$\;
			\For{$k=N,\ldots,1$}{
				$\lambda^j_{k-1} \gets 
(\mathrm{Id}- h \sum_{i=1}^l D_xF_i(x_{k-1}^j) 
u_{i,k})^{-T}\lambda_k^j$\;
				}			
			}
		}
		$x^{\mathrm{old}}\gets x$, 
		$u^{\mathrm{old}}\gets u$,
		$\lambda^{\mathrm{old}}\gets \lambda$
		\tcp*[r]{Set recovery point}
		\For(\tcp*[f]{Update of control and trajectories}){$k=1,\ldots,N$}{
	\For(\tcp*[f]{Correction of covectors}){$j=1,\ldots,M$}{	
	$\lambda^j_{k-1}\gets \lambda^j_{k-1}
	+ \frac1M \nabla a(x^{j,\mathrm{old}}_{k-1}-\Psi(x_0^j))
	- \frac1M \nabla a(x^{j}_{k-1}-\Psi(x_0^j))$\;}
	$u_k\gets \mathrm{argmax}_v\left(\mathcal{H}(x_{k-1},\lambda_{k-1},v) - \frac{1}{2\gamma}|v-u_k^{\mathrm{old}}|^2\right)$\;
	\For(\tcp*[f]{Trajectories in the
	next node}){$j=1,\ldots,M$}{
		$x^j_k \gets x^j_{k-1} + h\sum_{i=1}^lF_i(x_{k-1}^j)u_{i,k}$\;
		}
	}
	$\mathrm{Cost^{new}}
\gets \frac1M \sum_{j=1}^M
a(x^{j}_N - \Psi(x^j_0)) + \frac\beta2 ||u||_{L^2}^2$\;	
	\eIf(\tcp*[f]{Backtracking for $\gamma$}){$\mathrm{Cost}> \mathrm{Cost^{new}}$ }{
		$\mathrm{Cost}\gets \mathrm{Cost^{new}}$\;
		$\mathrm{flag} \gets 1$\;
		}
		{
		$x\gets x^{\mathrm{old}}$, 
		$u\gets u^{\mathrm{old}}$,
		$\lambda\gets \lambda^{\mathrm{old}}$
		\tcp*[r]{Recover saved quantities}
		$\gamma \gets \tau \gamma$\;
		$\mathrm{flag} \gets 0$\;
		}
	}
\caption{Training with Maximum Principle}
\label{alg:max_princ}
\end{algorithm}

\begin{remark}
In lines~22--24 of Algorithm~\ref{alg:max_princ}
we introduced a corrective term for the covectors.
This is due to the fact that
in \cite{SS80} where considered only problems without
end-point cost.\\
The maximization problem at line~25 can be solved 
with an explicit formula. Indeed, we have that
\[
u_{i,k} \gets \frac{1}{1+\gamma\beta}\left(
u_{i,k}^{\mathrm{old}} + \gamma\sum_{j=1}^M
\lambda_k^j \cdot F_i(x^j_k)
\right).
\]
We stress that the
 existence of such a simple expression is due to 
the fact that control system 
\eqref{eq:subr_sys_notat} is linear in the 
control variables.
\end{remark}
\begin{remark} 
We observe that the computation of the approximate
solutions of \eqref{eq:ode_covec_PMP} can be carried
out in parallel with respect to the index variable
of the elements of the ensembles (see lines~13--18 of
Algorithm~\ref{alg:max_princ}). 
However, when solving \eqref{eq:ode_traj_PMP}
in order to update the trajectories, the 
{\it for} loop on the elements of the ensemble
is nested into the {\it for} loop on the 
discretization nodes (see lines~21--29 of
Algorithm~\ref{alg:max_princ}).
\end{remark}

\end{section}

\begin{section}{Numerical Experiments}
\label{sec:num_exp}
In this section we describe the numerical experiments 
involving the approximation of an unknown flow by
means of Algorithm~\ref{alg:proj_grad_flow}
and Algorithm~\ref{alg:max_princ}.
{We implemented the codes in Matlab
and we ran them on  a laptop with
16 GB of RAM and a 6-core 2.20 GHz processor. 
Since we consider $\R^2$ as ambient space,
Theorem~\ref{thm:univ_app_flows} and
Theorem~\ref{thm:sys_contr} guarantee that the
linear-control system associated to the fields 
\[
F_1(x) := \frac{\partial}{\partial x_1}, \quad 
F_2(x) := \frac{\partial}{\partial x_2},
\]
\[
F_1'(x) := e^{-\frac{1}{2\nu}|x|^2}\frac{\partial}{\partial x_1}, \quad 
F_2'(x) := e^{-\frac{1}{2\nu}|x|^2}\frac{\partial}{\partial x_2}, 
\]
is capable of approximating on compact sets
diffeomorphisms that are diffeotopic to the 
identity. 
However, it looks natural to include
in the set of the controlled vector fields also 
the following ones:
\[
G_1^1(x) := x_1\frac{\partial}{\partial x_1}, \quad
G_1^2(x) := x_2\frac{\partial}{\partial x_1}, 
\]
\[
G_2^1(x) := x_1\frac{\partial}{\partial x_2}, \quad
G_2^2(x) := x_2\frac{\partial}{\partial x_2}.
\]
Indeed, with this choice, we observe that the 
corresponding control system
\begin{equation} \label{eq:ctrl_sys_1_exper}
\dot x = 
\left(
\begin{matrix}
u_1\\
u_2
\end{matrix}
\right)
+ e^{-\frac{1}{2\nu}|x|^2}
\left(
\begin{matrix}
u_1'\\
u_2'
\end{matrix}
\right)
+ \left(
\begin{matrix}
u_1^1 & u^2_1\\
u_2^1 & u_2^2
\end{matrix}
\right)
\left(
\begin{matrix}
x_1\\
x_2
\end{matrix}
\right)
\end{equation}
is capable of reproducing \textit{exactly} 
non-autonomous vector fields that are linear
in the state variables $(x_1,x_2)$.
Moreover,
the discretization of the control system
\eqref{eq:ctrl_sys_1_exper} on the evolution interval
$[0,1]$ with step-size $h=\frac1N$ gives rise to
a ResNet $\Phi = \Phi_N\circ \ldots\circ\Phi_1$
with $N$ layers, whose building blocks
have the form:
\begin{equation} \label{eq:block_exp_1}
\Phi_k(x) = x + h 
\left[
\left(
\begin{matrix}
u_1\\
u_2
\end{matrix}
\right)
+ e^{-\frac{1}{2\nu}|x|^2}
\left(
\begin{matrix}
u_1'\\
u_2'
\end{matrix}
\right)
+ \left(
\begin{matrix}
u_1^1 & u^2_1\\
u_2^1 & u_2^2
\end{matrix}
\right)
\left(
\begin{matrix}
x_1\\
x_2
\end{matrix}
\right)
\right],
\end{equation}
and each of them has $8$ parameters. 
}

\subsection{Approximation of a diffeomorphism}
We consider the following diffeomorphism 
$\tilde \Psi: \R^2 \to\R^2$:
\[
\tilde \Psi(x) := x + 
\left( \begin{matrix}
2x_1e^{x_1^2-1} \\
2x_2^3
\end{matrix} \right) 
+
\left( \begin{matrix}
-4 \\
-4.5
\end{matrix} \right),
\] 
the rotation $R:\R^2\to\R^2$ centered at the
origin and with angle $\pi/3$, and the
translation $T:\R^2\to\R^2$ prescribed by the vector
$(0.3,0.2)$. Finally, we set
\begin{equation} \label{eq:Psi_exp}
\Psi := \tilde \Psi\circ T \circ R.
\end{equation}
We generate the training dataset
$\{ x^1,\ldots,x^M \}$ with $M=900$ points
by constructing a uniform grid in the 
square centered at the origin and with
side of length $\ell=1.5$.
In Figure~\ref{fig:traing_data_transf}
we report the training dataset and its image
trough $\Psi$.
\begin{center}
\begin{figure}
\includegraphics[scale=0.49]{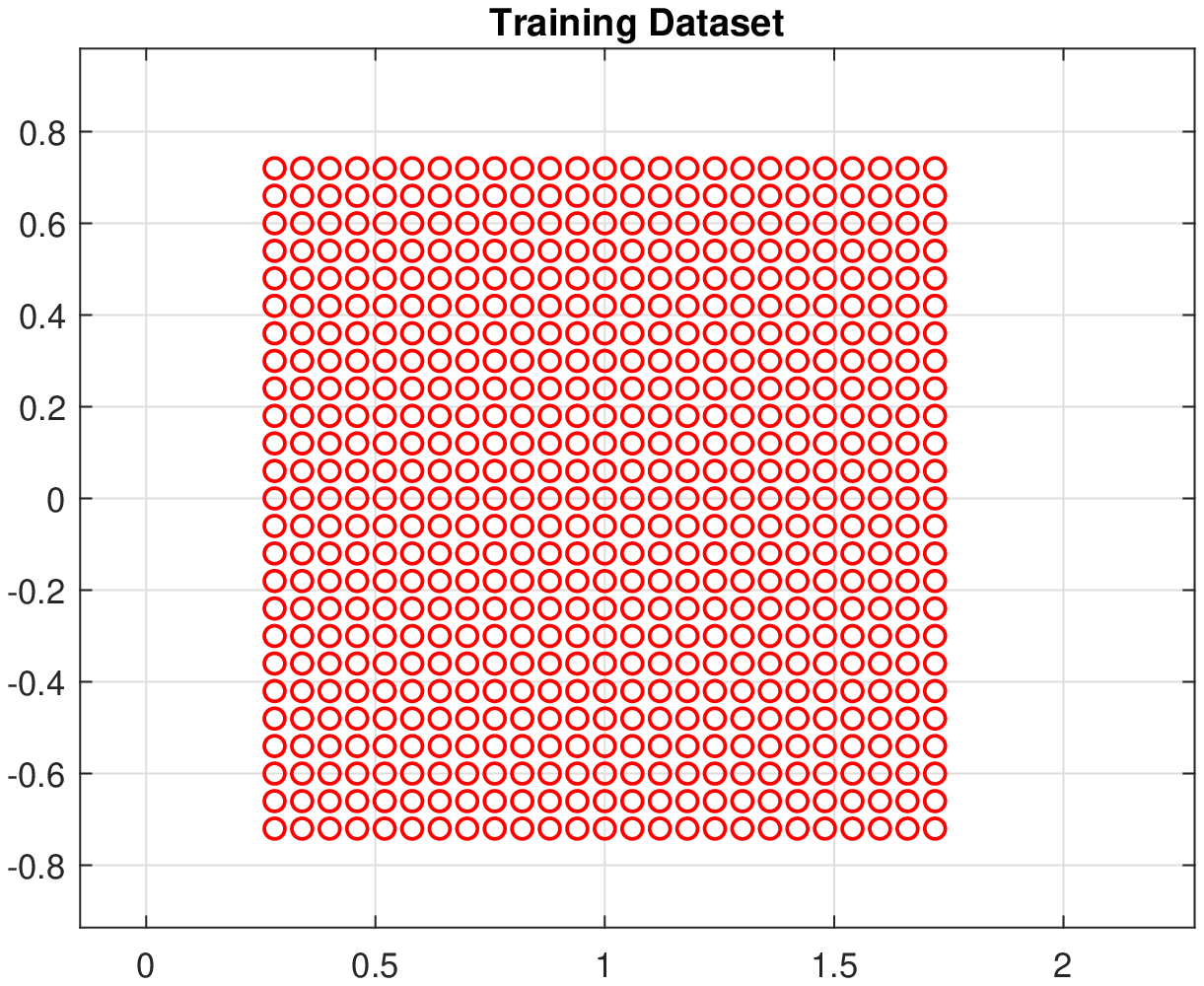}
\includegraphics[scale=0.49]{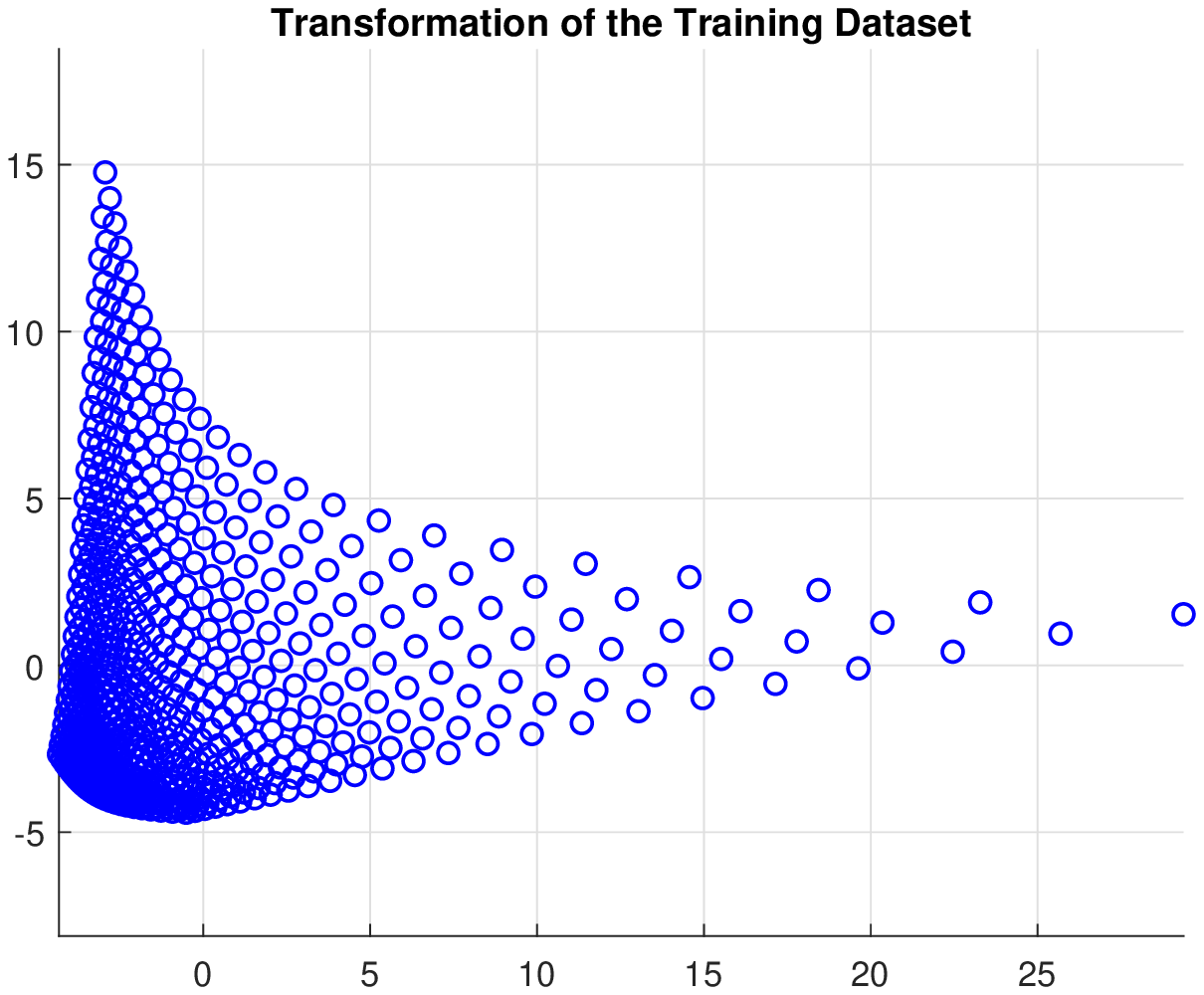}
\caption{On the left we report the
grid of points $\{ x^1,\ldots,x^M \}$
where we have evaluated the 
diffeomorphism $\Psi:\R^2\to\R^2$ 
defined as in \eqref{eq:Psi_exp}. The picture on the
right represents the transformation of the training
dataset through the diffeomorphism $\Psi$. 
} \label{fig:traing_data_transf}
\end{figure}
\end{center}
If we denote by $\mu$ the probability measure that
charges uniformly the square and if we set
$\mu_M:= \frac1M \sum_{j=1}^M\delta_{x^j}$, we 
obtain the following estimate 
\begin{equation*}
W_1(\mu_M,\mu) \leq \frac{\sqrt2 \ell}{2\sqrt{M}},
\end{equation*}
that can be used to compute the \textit{a priori}
estimate of the generalization error 
provided by \eqref{eq:est_gen_err_univers}.
We use the $1-$Lipschitz loss function 
$$a(x-y):=\sqrt{1+(x_1-y_1)^2+(x_2-y_2)^2}-1,$$
and we look for a minimizer of 
\begin{equation} \label{eq:obj_exp}
\F^M(u) := \frac{1}{900}\sum_{j=1}^{900}
a(\Phi_u(x^j)-\Psi(x^j))
+ \frac{\beta}{2}||u||^2_{L^2},
\end{equation}
where $\beta >0$ is the 
regularization hyper-parameter. 
In the training phase we use the same dataset
for Algorithm~\ref{alg:proj_grad_flow} and
Algorithm~\ref{alg:max_princ}, and in both cases
the initial guess of the control is $u\equiv 0$.
Finally, the testing dataset has been generated by 
randomly sampling $300$ points using $\mu$,
the uniform probability measure on the square.
The value of the hyper-parameter
$\nu$ is set equal to $20$.
We first tried to approximate the diffeomorphism
$\Psi$ using $h=2^{-4}$, resulting in $16$ inner
layers. 
Hence, recalling that each building-block
 \eqref{eq:block_exp_1} has
$8$ parameters,
the corresponding ResNet has 
in total $128$ parameters.
 We tested different values of $\beta$, and 
we set $\mathrm{max}_{iter}=500$.
The results obtained by 
Algorithm~\ref{alg:proj_grad_flow} and
Algorithm~\ref{alg:max_princ} are reported in
Table~1
and Table~2, respectively.
We observe that in both algorithms  the
Lipschitz constant of the produced diffeomorpism
grows as the 
hyper-parameter $\beta$ gets smaller,
 consistently with the theoretical intuition.
As regards the testing error, we observe that 
it always remains reasonably close to the 
corresponding training error. 

\begin{table}
\begin{tabular}{|r|r|r|r|}
\hline
$\beta$ & $L_{\Phi_u}$ & Training error & Testing error \\
\hline
$10^0$ & $1.19$ & $3.8785$ & $3.8173$\\
\hline
$10^{-1}$ & $8.40$ & $1.3143$ & $1.2476$\\
\hline
$10^{-2}$ & $9.32$ & $1.1991$ & $1.1451$\\
\hline
$10^{-3}$ & $9.37$ & $1.1852$ & $1.1330$\\
\hline
$10^{-4}$ & $9.37$ & $1.1839$ & $1.1318$\\
\hline
\end{tabular}
\vspace{0.3cm}
\caption{ResNet~\ref{eq:block_exp_1}, $16$ layers,
$128$ parameters,
Algorithm~\ref{alg:proj_grad_flow}.
Running time $\sim 160$ s.
}
\end{table}
\begin{table}
\begin{tabular}{|r|r|r|r|}
\hline
$\beta$ & $L_{\Phi_u}$ & Training error & Testing error \\
\hline
$10^0$ & $1.19$ & $3.8749$ & $3.8157$\\
\hline
$10^{-1}$ & $8.40$ & $1.3084$ & $1.2455$\\
\hline
$10^{-2}$ & $9.32$ & $1.2014$ & $1.1486$\\
\hline
$10^{-3}$ & $9.33$ & $1.1898$ & $1.1387$\\
\hline
$10^{-4}$ & $9.33$ & $1.1898$ & $1.1379$\\
\hline
\end{tabular}
\vspace{0.3cm}
\caption{ResNet~\ref{eq:block_exp_1}, $16$ layers,
$128$ parameters, Algorithm~\ref{alg:max_princ}.
Running time $\sim 130$ s.
}
\end{table}

\noindent
We report in Figure~\ref{fig:Results_best_1}
the image of the 
approximation that achieves the best training
and testing errors, namely 
Algorithm~\ref{alg:proj_grad_flow} 
with $\beta=10^{-4}$.
As we may observe, the prediction is quite 
unsatisfactory, both on the training and on the
testing data-sets. 
Finally, we report that the formula 
\eqref{eq:est_gen_err_univers} correctly
provides an upper bound to the testing error,
even though it is too pessimistic to be of practical
use.

\begin{center}
\begin{figure}
\includegraphics[scale=0.49]{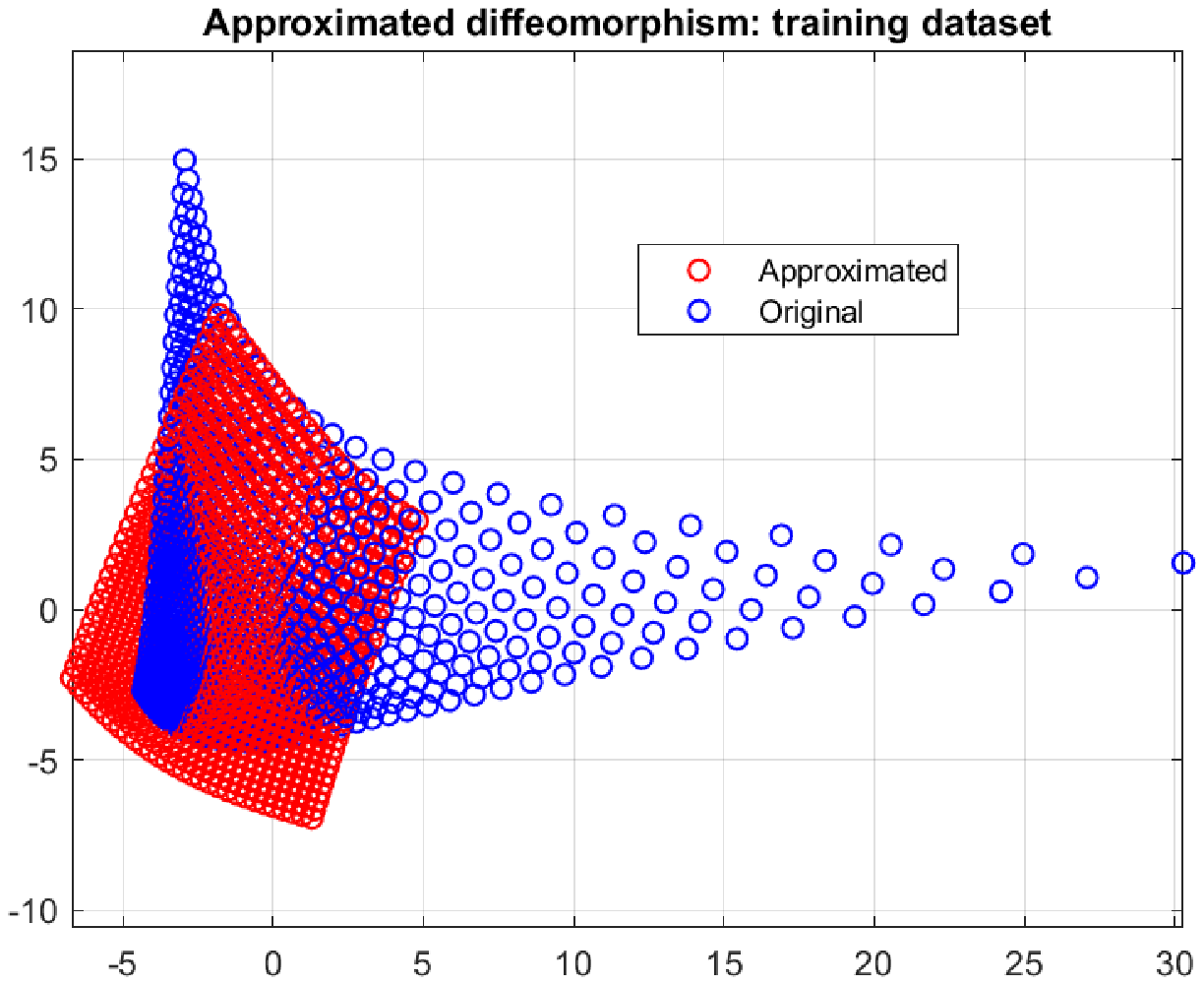}
\includegraphics[scale=0.49]{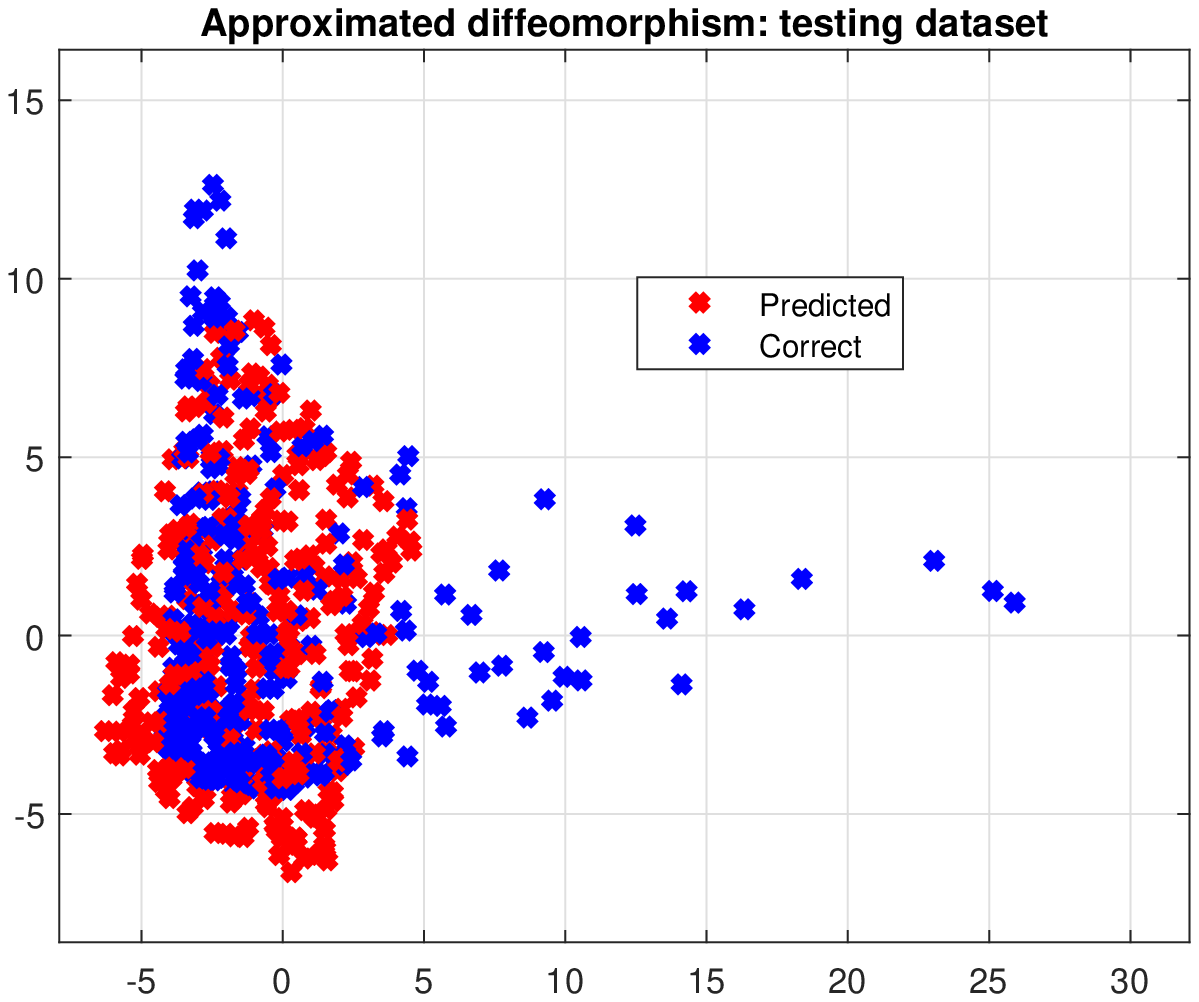}\\
\includegraphics[scale=0.49]{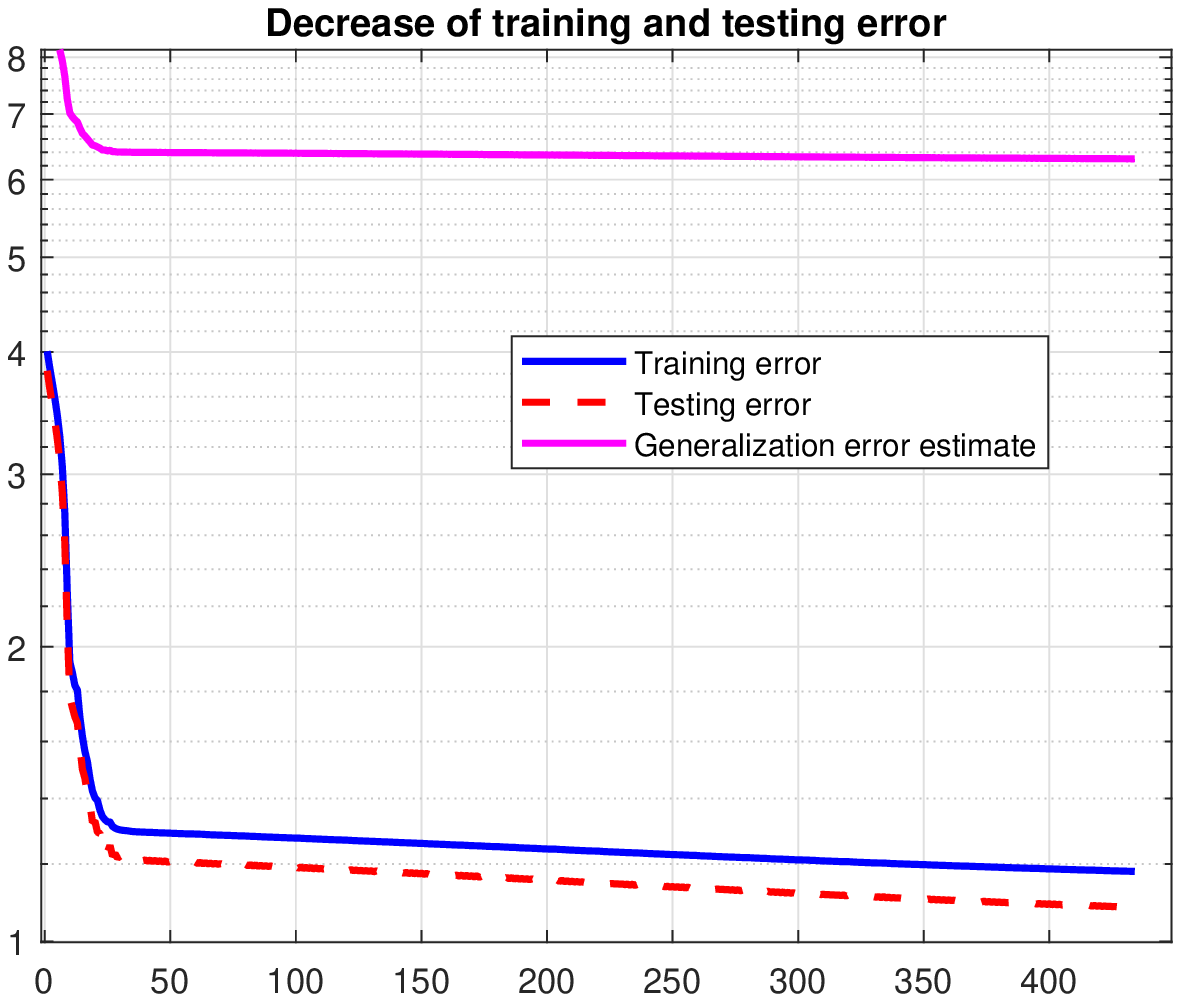}
\caption{ResNet~\ref{eq:block_exp_1}, 16 layers,
Algorithm~\ref{alg:proj_grad_flow}, $\beta=10^{-4}$.
On the top-left we reported the 
transformation of the initial grid through the 
approximating diffeomorphism
(red circles) and through 
the original one (blue circles).
On the top-right, we plotted the prediction 
on the testing data-set provided by the 
approximating diffeomorphism (red crosses) and the
correct values obtained through the original
transformation (blue crosses). In both cases, the
approximation obtained is unsatisfactory.
At bottom we plotted the decrease of the 
training error and the testing error versus
the number of iterations. Finally, the curve in 
magenta represents the estimate of the generalization
error provided by \eqref{eq:est_gen_err_univers}.
} \label{fig:Results_best_1}
\end{figure}
\end{center}

In order to improve the quality of the approximation,
a natural attempt consists in trying to increase the
depth of the ResNet. Therefore, we repeated the
experiments setting $h=2^{-5}$, that corresponds
to $32$ layers. Recalling that the ResNet in exam
has $8$ parameters per layer, the architecture has
globally $256$ weights.
The results are reported in Table~3 and Table~4.
Unfortunately, despite doubling the depth of the 
ResNet, we do not observe any relevant improvement
in the training nor in the testing error.
Using the idea explained in 
Remark~\ref{rmk:sys_fields}, instead of 
further increase the number of the layers, 
we try to enlarge the family of the
controlled vector fields in the control system
associated to the ResNet.

\begin{table}
\begin{tabular}{|r|r|r|r|}
\hline
$\beta$ & $L_{\Phi_u}$ & Training error & Testing error \\
\hline
$10^0$ & $1.19$ & $3.8779$ & $3.8168$\\
\hline
$10^{-1}$ & $8.40$ & $1.3074$ & $1.2425$\\
\hline
$10^{-2}$ & $9.26$ & $1.2015$ & $1.1477$\\
\hline
$10^{-3}$ & $9.34$ & $1.1860$ & $1.1352$\\
\hline
$10^{-4}$ & $9.34$ & $1.1842$ & $1.1332$\\
\hline
\end{tabular}
\vspace{0.3cm}
\caption{ResNet~\ref{eq:block_exp_1}, $32$ layers,
$256$ parameters,
Algorithm~\ref{alg:proj_grad_flow}.
Running time $\sim 320$ s.}
\end{table}
\begin{table}
\begin{tabular}{|r|r|r|r|}
\hline
$\beta$ & $L_{\Phi_u}$ & Training error & Testing error \\
\hline
$10^0$ & $1.19$ & $3.8739$ & $3.8148$\\
\hline
$10^{-1}$ & $8.35$ & $1.3085$ & $1.2449$\\
\hline
$10^{-2}$ & $9.23$ & $1.2075$ & $1.1538$\\
\hline
$10^{-3}$ & $9.26$ & $1.1931$ & $1.1416$\\
\hline
$10^{-4}$ & $9.26$ & $1.1918$ & $1.1404$\\
\hline
\end{tabular}
\vspace{0.3cm}
\caption{ResNet~\ref{eq:block_exp_1}, $32$ layers,
$256$ parameters, Algorithm~\ref{alg:max_princ}.
Running time $\sim 260$ s.
}
\end{table} 

\subsection{Enlarged family of controlled fields}
Using the ideas expressed in 
Remark~\ref{rmk:sys_fields}, we enrich the 
family of the controlled fields.
In particular, in addition to the fields considered
above, we include the following ones:
\[
G^{1,1}_1:= x_1^2e^{-\frac{1}{2\nu}|x|^2}\frac{\partial}{\partial x_1}, \quad
G^{1,2}_1:= x_1x_2e^{-\frac{1}{2\nu}|x|^2}\frac{\partial}{\partial x_1}, \quad
G^{2,2}_1:=x_2^2e^{-\frac{1}{2\nu}|x|^2}\frac{\partial}{\partial x_1}, 
\]
\[
G^{1,1}_2:= x_1^2e^{-\frac{1}{2\nu}|x|^2}\frac{\partial}{\partial x_2}, \quad
G^{1,2}_2:= x_1x_2e^{-\frac{1}{2\nu}|x|^2}\frac{\partial}{\partial x_2}, \quad
G^{2,2}_2:=x_2^2e^{-\frac{1}{2\nu}|x|^2}\frac{\partial}{\partial x_2}.
\]
Therefore, the resulting linear-control system
on the time interval $[0,1]$ has the form
\begin{align*} 
\dot x = 
\left(
\begin{matrix}
u_1\\
u_2
\end{matrix}
\right)&
+ e^{-\frac{1}{2\nu}|x|^2}
\left(
\begin{matrix}
u_1'\\
u_2'
\end{matrix}
\right)
+ \left(
\begin{matrix}
u_1^1 & u^2_1\\
u_2^1 & u_2^2
\end{matrix}
\right)
\left(
\begin{matrix}
x_1\\
x_2
\end{matrix}
\right)
\\ & \qquad+
e^{-\frac{1}{2\nu}|x|^2}
\left(
\begin{matrix}
u_1^{1,1}x_1^2 + u_1^{1,2}x_1x_2 + u_1^{2,2}x_2^2\\
u_2^{1,1}x_1^2 + u_2^{1,2}x_1x_2 + u_2^{2,2}x_2^2
\end{matrix}
\right),
\end{align*}
while the building blocks of the corresponding ResNet
have the following expression:
\begin{align}
\Phi_k(x) = x + h & \label{eq:block_exp_2_1}
\left[
\left(
\begin{matrix}
u_1\\
u_2
\end{matrix}
\right)
+ e^{-\frac{1}{2\nu}|x|^2}
\left(
\begin{matrix}
u_1'\\
u_2'
\end{matrix}
\right)
+ \left(
\begin{matrix}
u_1^1 & u^2_1\\
u_2^1 & u_2^2
\end{matrix}
\right)
\left(
\begin{matrix}
x_1\\
x_2
\end{matrix}
\right)
\right.\\
& \left. \qquad+ \label{eq:block_exp_2_2}
e^{-\frac{1}{2\nu}|x|^2}
\left(
\begin{matrix}
u_1^{1,1}x_1^2 + u_1^{1,2}x_1x_2 + u_1^{2,2}x_2^2\\
u_2^{1,1}x_1^2 + u_2^{1,2}x_1x_2 + u_2^{2,2}x_2^2
\end{matrix}
\right) \right]
\end{align}
for $k=1,\ldots,N$,
where $h=\frac1N$ is the discretization step-size
and $N$ is the number of layers of the ResNet.
We observe that each building block has
$14$ parameters.

As before, we set $\nu = 20$,
$\mathrm{max}_{iter}=500$ and we  considered 
$h=2^{-4}$, resulting in a ResNet with $16$ layers
and with total number of weights equal to
$224$.
We used the same training data-set as above, 
namely the grid of points and the corresponding
image trough $\Psi$ depicted in 
Figure~\ref{fig:traing_data_transf}.
We trained the network using both
Algorithm~\ref{alg:proj_grad_flow} 
and Algorithm~\ref{alg:max_princ}.
The results are collected in Table~5 and Table~6,
respectively. Once again, we observe that 
the Lipschitz constant of the approximating 
diffeomorphisms grows as $\beta$ is reduced.
In this case, with both algorithms, 
the training and testing errors are much lower if
compared with the best case of the ResNet~\ref{eq:block_exp_1}. We insist on the fact that
in the present case the ResNet~\ref{eq:block_exp_2_1}-\ref{eq:block_exp_2_2}
has in total $224$ parameters divided into $16$
layers, and it overperforms the 
ResNet~\ref{eq:block_exp_1} with $256$ parameters
divided into $32$ layers.
We report in Figure~\ref{fig:Results_best_2} the approximation produced 
by Algorithm~\ref{alg:proj_grad_flow} with
$\beta = 10^{-3}$.
In this case the approximation provided is 
very satisfactory, and we observe that it is better
in the area where more observations are available.
Finally, also in this case the estimate on
the expected generalization error
\eqref{eq:est_gen_err_univers} provides
an upper bound for the testing error, but 
at the current state it is too coarse to
be of practical use.

\begin{table}
\begin{tabular}{|r|r|r|r|}
\hline
$\beta$ & $L_{\Phi_u}$ & Training error & Testing error \\
\hline
$10^0$ & $10.14$ & $2.3791$ & $2.3036$\\
\hline
$10^{-1}$ & $13.84$ & $0.1809$ & $0.2314$\\
\hline
$10^{-2}$ & $15.64$ & $0.1290$ & $0.1784$\\
\hline
$10^{-3}$ & $15.83$ & $0.1254$ & $0.1747$\\
\hline
$10^{-4}$ & $15.86$ & $0.1257$ & $0.1751$\\
\hline
\end{tabular}
\vspace{0.3cm}
\caption{ResNet~\ref{eq:block_exp_2_1}-\ref{eq:block_exp_2_2}, $16$ layers,
$224$ parameters,
Algorithm~\ref{alg:proj_grad_flow}.
Running time $\sim 320$ s.
}
\end{table}
\begin{table}
\begin{tabular}{|r|r|r|r|}
\hline
$\beta$ & $L_{\Phi_u}$ & Training error & Testing error \\
\hline
$10^0$ & $10.78$ & $2.3638$ & $2.3910$\\
\hline
$10^{-1}$ & $14.32$ & $0.1921$ & $0.2422$\\
\hline
$10^{-2}$ & $15.43$ & $0.1887$ & $0.2347$\\
\hline
$10^{-3}$ & $15.56$ & $0.2260$ & $0.2719$\\
\hline
$10^{-4}$ & $15.59$ & $0.2127$ & $0.2564$\\
\hline
\end{tabular}
\vspace{0.3cm}
\caption{ResNet~\ref{eq:block_exp_2_1}-\ref{eq:block_exp_2_2}, $16$ layers,
$224$ parameters, Algorithm~\ref{alg:max_princ}.
Running time $\sim 310$ s.}
\end{table} 

\begin{center}
\begin{figure}
\includegraphics[scale=0.49]{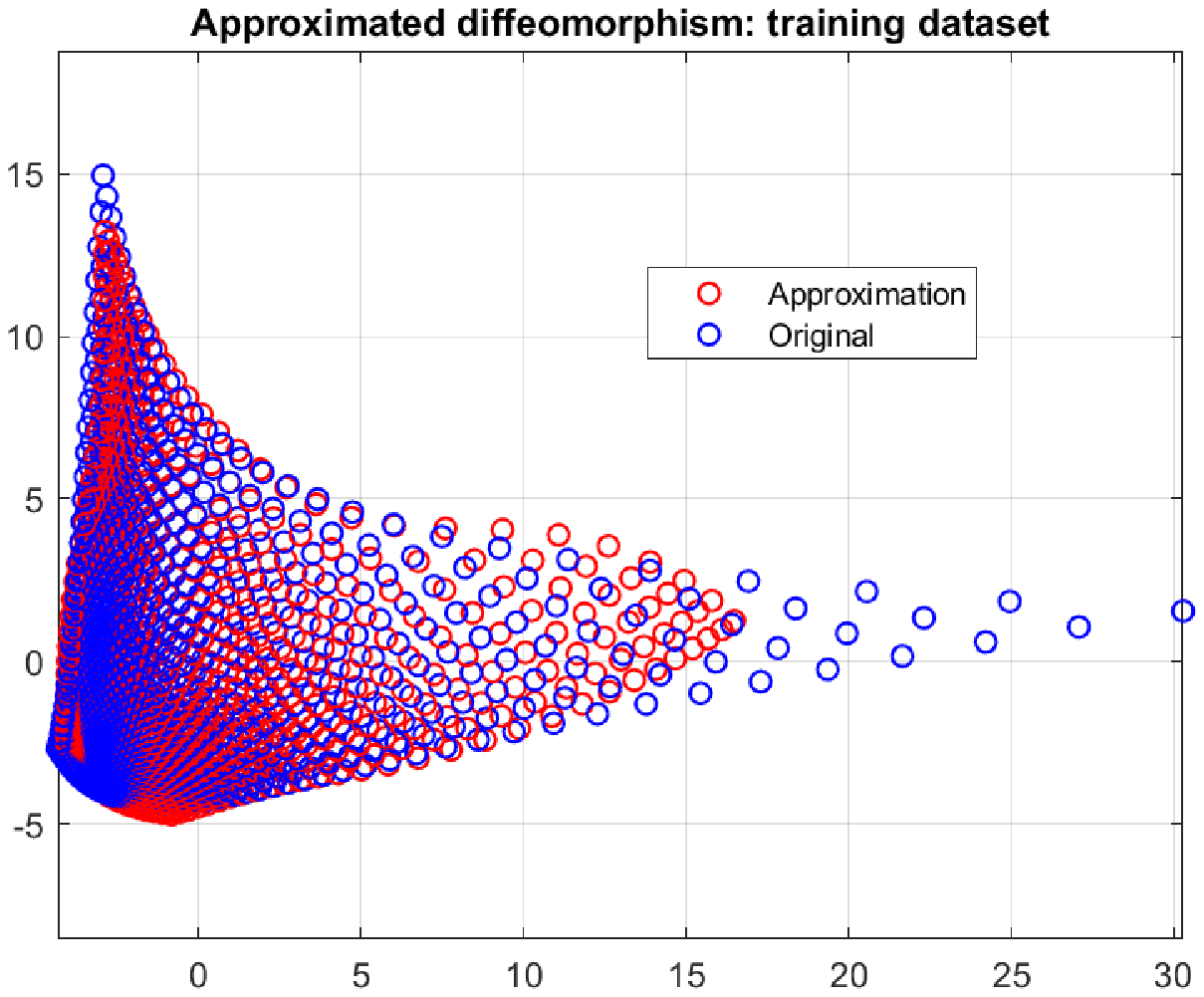}
\includegraphics[scale=0.49]{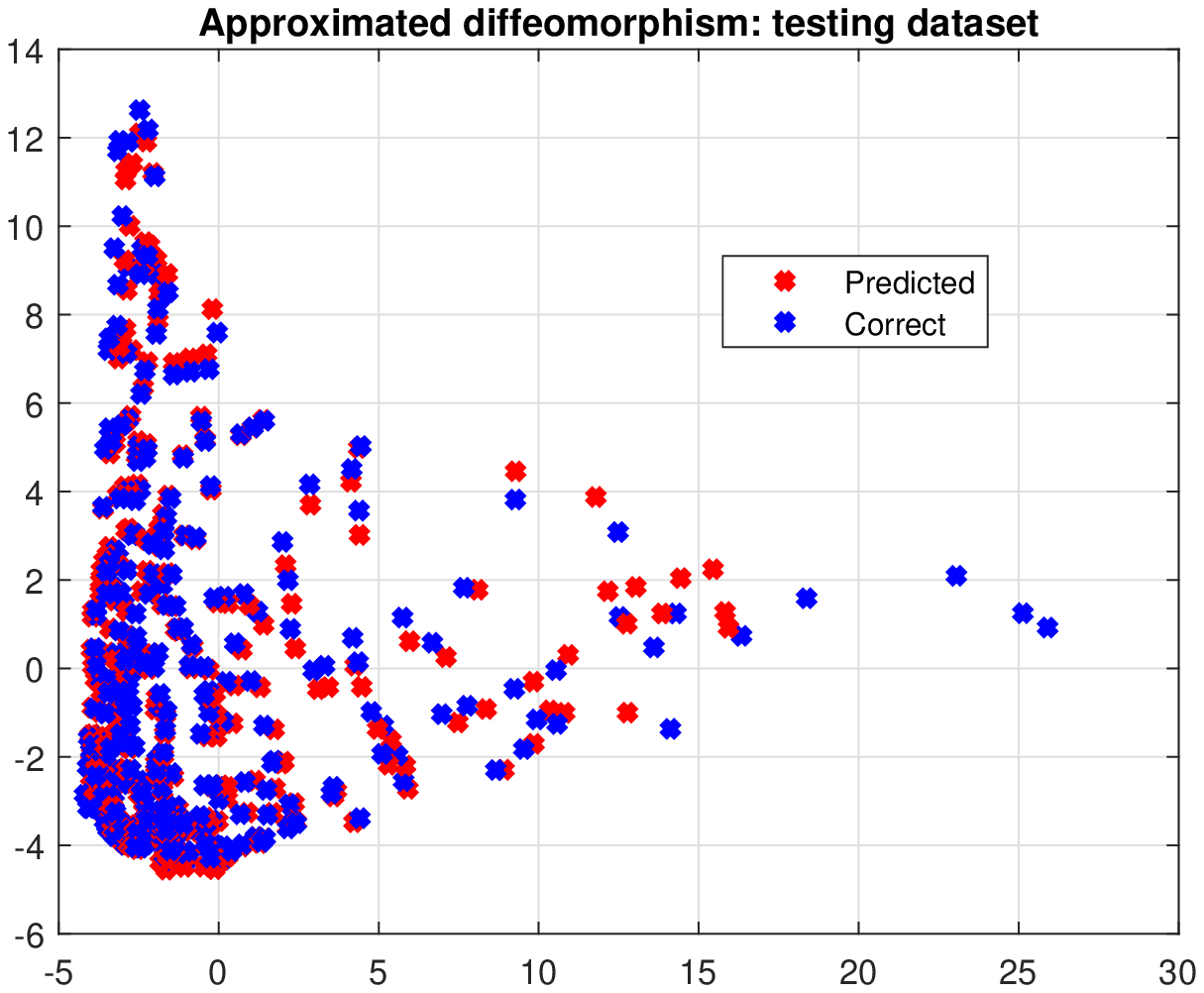}\\
\includegraphics[scale=0.49]{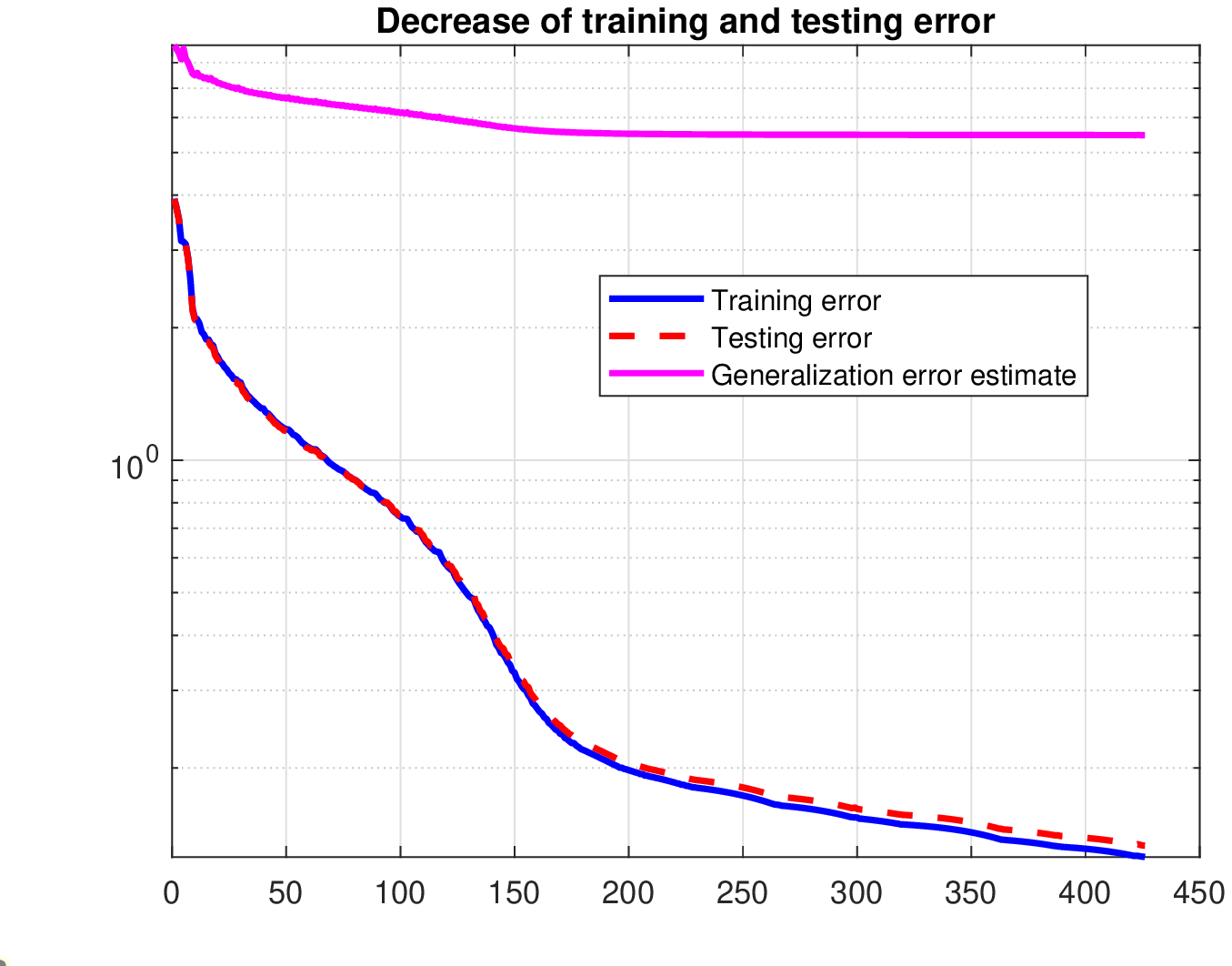}
\caption{ResNet~\ref{eq:block_exp_2_1}-\ref{eq:block_exp_2_2}, 16 layers,
Algorithm~\ref{alg:proj_grad_flow}, $\beta=10^{-3}$.
On the top-left we reported the 
transformation of the initial grid through the 
approximating diffeomorphism
(red circles) and through 
the original one (blue circles).
On the top-right, we plotted the prediction 
on the testing data-set provided by the 
approximating diffeomorphism (red crosses) and the
correct values obtained through the original
transformation (blue crosses). In both cases, the
approximation obtained is good, and we observe that
it is better where we have more data density.
At bottom we plotted the decrease of the 
training error and the testing error versus
the number of iterations. Finally, the curve in 
magenta represents the estimate of the generalization
error provided by \eqref{eq:est_gen_err_univers}.
} \label{fig:Results_best_2}
\end{figure}
\end{center}

\end{section}

\section*{Conclusions}

In this paper we have derived a
Deep Learning architecture 
for a Residual Neural Network 
starting from a control system.
Even though this approach has already been 
undertaken in some recent works (see
\cite{E17}, \cite{HR17}, \cite{LCTE17},
\cite{BC19}, \cite{BCFH21}),
the original aspect of our contribution lies in
the fact that the control system considered here
is linear in the control variables.
In spite of this apparent simplicity, the flows
generated by these control systems (under a suitable
assumption on the controlled vector fields)
can approximate with arbitrary precision
on compact sets any diffeomorphism diffeotopic
to the identity (see \cite{AS1}, \cite{AS2}).
Moreover, the linearity in the control variables
simplifies the implementation and 
the computational effort required in the training
procedure. Indeed, when the parameters are learned
via a Maximum-Principle approach, the maximization
of the Hamiltonian function is typically highly
time-consuming if the control system is
non-linear in the control variables.
In the paper we propose two training procedure.
The first one consists in projecting onto a 
finite-dimensional space the gradient flow 
associated to a proper Optimal Control problem.
The second procedure relies on
an iterative method for the solution of Optimal 
Control problems 
based on Pontryagin Maximum Principle.
We have tested the learning algorithms on 
a diffeomorphism approximation problem in 
$\R^2$, obtaining interesting results.
We plan to investigate in future work
the performances of this kind of Neural Networks 
in other learning tasks, such as classification.

\subsection*{Acknowledgments}
The Author acknowledges partial support from
INDAM-GNAMPA. 
The Author wants to thank Prof. A. Agrachev and 
Prof. A. Sarychev for encouraging and for the helpful 
discussions.
{The Author is grateful to an anonymous
referee for the invaluable suggestions that
contributed to improve the quality of the paper.}


\vspace{1cm}

\noindent
(A. Scagliotti) Scuola Internazionale Superiore di
Studi Avanzati, Trieste, Italy.\\
{\it Email address: }{\texttt{ascaglio@sissa.it}}

\end{document}